\documentclass[11pt]{article}
\usepackage{amsthm,amsfonts,amsmath,amssymb}

\usepackage{amsthm,amsfonts,amsmath, amscd}
\swapnumbers
\theoremstyle{plain}

\newtheorem{thm}{THEOREM}[section]
\newtheorem{lm}[thm]{LEMMA}
\newtheorem{cl}[thm]{COROLLARY}
\newtheorem{prop}[thm]{PROPOSITION}
\theoremstyle{definition}
\newtheorem{defi}[thm]{DEFINITION}
\theoremstyle{definition}
\newtheorem{remark}[thm]{Remark}
\newcommand{\upchi}{\raise1pt\hbox{$\chi$}}
\newcommand{\R}{{\mathord{\mathbb R}}}

\renewcommand{\a}{\alpha}
\renewcommand{\b}{\beta}

\newcommand{\F}{{\mathcal{F}}}

\newcommand{\D}{{\mathcal{D}}}

\renewcommand{\|}{{\Vert}}
\numberwithin{equation}{section}
\def\dd{{\rm d}}
\def\H{{\mathcal H}}
\def\W{{\rm  W}}

\begin{document}

\markboth{\scriptsize{EACAF \today}}{\scriptsize{EACAF \today}}

\title{Stability for a GNS inequality and the Log-HLS inequality,\\
with
application to the critical mass Keller-Segel equation}

\author{\vspace{5pt} Eric A. Carlen$^1$ and 
Alessio Figalli$^{2}$   \\
\vspace{5pt}\small{$1.$ Department of Mathematics, Hill Center,}\\[-6pt]
\small{Rutgers University,
110 Frelinghuysen Road
Piscataway NJ 08854-8019 USA}\\
\vspace{5pt}\small{$2.$ Department of Mathematics, The University of Texas at Austin,}\\[-6pt]
\small{1 University Station C1200, Austin TX 78712 USA}
 }
 
\maketitle 
\footnotetext                                                                         
[1]{Work partially
supported by U.S. National Science Foundation
grant DMS 0901632.    }                           
                           \footnotetext
[2]{Work partially
supported by U.S. National Science Foundation
grant DMS 0969962.\\
\copyright\, 2011 by the authors. This paper may be
reproduced, in its
entirety, for non-commercial purposes.}

\maketitle

\begin{abstract}
Starting from the quantitative stability result of Bianchi and Egnell
for the $2$-Sobolev inequality,
we deduce several different stability results for a Gagliardo-Nirenberg-Sobolev inequality in the plane.
Then, exploiting the connection between this inequality and a fast diffusion equation,
we get a quantitative stability for the Log-HLS inequality.
Finally, using all these estimates, we prove 
a quantitative convergence result for the critical mass Keller-Segel system.
\end{abstract}

\section{Introduction}  
Let $W^{1,2}(\R^n)$ denote the space of measurable functions on $\R^n$  that have a square integrable distributional gradient. 
The Gagliardo-Nirenberg-Sobolev (GNS) inequality states that,
for $n \geq 2$  and all $1\leq p \leq q < r(n)$ (with $r(2):=\infty$, and $r(n) := 2n/(n-2)$ if $n\geq 3$),  there is a finite constant $C$ such that  for  all $u\in W^{1,2}(\R^n)$,  
\begin{equation}\label{gnsgen}
\|u\|_q \leq  C\|u\|_p^{1-\theta}\|\nabla u\|_2^\theta\ 
\end{equation}
where
\begin{equation}\label{gnsgen2}
\frac{1}{q} = \frac{\theta}{r(n)} + \frac{1-\theta}{p}\ .
\end{equation}
For $n\geq 3$ (so that $r(n) < \infty$), (\ref{gnsgen}) is valid also for $q = r(n)$, in which case  (\ref{gnsgen2}) gives $\theta =1$
and  (\ref{gnsgen}) reduces to the Sobolev inequality
\begin{equation}\label{sharpsob}
\|u\|_{2n/(n-2)}^2 \leq S_n \|\nabla u\|_2^2\ ,
\end{equation}
for which the sharp constant $S_n$ is known.

The are a few other choices of the $p$ and $q$ for which sharp constants are known. For $p=1$ and $q=2$, 
 (\ref{gnsgen2}) gives $\theta =n/(n+2)$, so that  (\ref{gnsgen}) reduces to the sharp Nash inequality \cite{CarLos2}
 \begin{equation}\label{sharpNash}
\|u\|_{2}^2 \leq C_n \|\nabla u\|_2^{n/(n+2)}\|u\|_1^{2/(n+2)}\ .
\end{equation}
(This inequality is valid also for $n=1$, even though $r(1)$ is negative.)

More recently, the sharp constant has been found \cite{DD} for a one-parameter family of GNS inequalities for each $n\geq 2$:  For $t> 0$, let $p = t+1$,
and let $q=2t$. Then
\begin{equation}\label{gnsdd}
\|u\|_{2t} \leq A_{n,t} \|\nabla u\|_2^\theta \|u\|_{t+1}^{1-\theta}\,,\qquad
\theta=\frac{n(t-1)}{t[2n-(1+t)(n-2)]}\,.
\end{equation}
(This inequality is a trivial identity for $t=1$, and is valid even for $t<1/2$, in which case $p < 1$ so that strictly speaking, for $0<t<1/2$,
the sharp inequality is not included in (\ref{gnsgen}).)

It turns out that there is a close relation between the sharp Sobolev inequality (\ref{sharpsob}) and the family of GNS inequalities (\ref{gnsdd}).
One aspect of this is that the functions $u$ that saturate these inequalities are simply powers of one another:  The optimal constant
$S_n$ in  (\ref{sharpsob}) is given by  \cite{aub, ros, tal}
\begin{equation}\label{sharpsob2}
 S_n = \frac{\|v\|_{2n/(n-2)}^2}{ \|\nabla v\|_2^2} \qquad {\rm where}\qquad v(x) = (1+|x|^2)^{-(n-2)/2}\ ,
\end{equation}
and moreover, with this value of $S_n$, there is equality in  (\ref{sharpsob}) if and only if $u$ is a multiple of 
$v(\mu(x-x_0))$ for some $\mu>0$ and some $x_0\in \R^n$.

Likewise, for $t>1$ the optimal constant $A_{n,t}$  in (\ref{gnsdd}) is given by \cite{DD}
\begin{equation}\label{gnsdd2}
 A_{n,t} = \frac{ \|v\|_{2t}}{\|v\|_{t+1}^{1-\theta}  \|\nabla v\|_2^\theta} \qquad {\rm where}\qquad v(x) = (1+|x|^2)^{-1/(t-1)}\ ,
\end{equation}
and moreover, with this value of $A_{n,t}$, there is equality in  (\ref{sharpsob}) if and only if $u$ is a multiple of 
$v(\mu(x-x_0))$ for some $\mu>0$ and some $x_0\in \R^n$.  
However, this is a very particular feature of this family: the sharp Nash inequality has optimizers of an entirely different form; see \cite{CarLos2}. 

Another aspect of this close relation between (\ref{sharpsob}) and  (\ref{gnsdd}) is that both inequalities    can be proved using ideas 
coming from the theory of optimal mass transportation \cite{CNV}. To apply these kinds of ideas,
one should consider   
(\ref{sharpsob}) and  (\ref{gnsdd}) as inequalities for a mass density $\rho(x) := |u(x)|^q$.  
(Throughout this paper, by a {\em density} we mean a non-negative integrable function.)
More precisely, it turns out for $r\geq 1-1/n$,  that the functional
$$\rho\mapsto  \frac{1}{r-1}\int _{\R^n}\rho^{r}(x)\dd x$$
is convex along the {\em displacement interpolation} $\rho_t$, $0\leq t \leq 1$, between two densities $\rho_0$ and 
$\rho_1$ of the same mass on $\R^n$ \cite{McC}.  Taking $\rho_0(x) = v^q(x)$ with $v$ as above, and taking $\rho_1(x) = u^q(x)$ where 
$u$ is a non-negative function with $\|u\|_q = \|v\|_q$,
the ``above the tangent line inequality'' for convex functions translates into (\ref{sharpsob}) and  (\ref{gnsdd}), as shown in \cite{CNV}.

 In this paper we are concerned with the {\em stability} properties of the GNS inequalities (\ref{gnsdd}), and the applications of
 this stability to certain partial differential equations. In fact,
 because of its connection with the Keller-Segel equation that we consider here,
 we shall focus only on the $n=2$, $t=3$  case of  (\ref{gnsdd}).
 
This case may be written explicitly as
 \begin{equation}\label{ourgns}
 \pi \int_{\R^2}u^6(x)\dd x \leq  \left(\int_{\R^2}|\nabla u(x)|^2 \dd x\right)\left( \int_{\R^2}u^4(x)\dd x\right)\ ,
 \end{equation}
 where $u$, here and throughout the rest of the paper, is a non-negative function on $\R^2$.

\begin{defi}[GNS deficit  functional]\label{GNSdefdef} Given a  non-negative function $u$ in
$W^{1,2}(\R^2)$,
define $\delta_{{\rm GNS}}[u]$ by
\begin{equation}\label{rob6A}
\delta_{{\rm GNS}}[u] :=  \left( \int_{\R^2} |\nabla u|^2\dd y\right)^{1/2}\left(\int_{\R^2} u^4 \dd y \right)^{1/2} - 
\left(\pi \int_{\R^2} u^6\dd y \right)^{1/2}\ .
\end{equation} 
Also, for $\lambda>0$ and $x_0\in \R^2$, define 
\begin{equation}\label{rob6Z}
v_{\lambda,x_0} := (1 + \lambda^2|x-x_0|^2)^{-1/2}\ .
\end{equation}
Also, throughout the paper, we use $v(x)$ to denote the function $v_{1,0}$; i.e.,
\begin{equation}\label{rob6ZZ}
v(x) := (1+ |x|^2)^{-1/2}\ .
\end{equation}
\end{defi}
 
 By the theorem \cite{DD}  of Del Pino and Dolbeault,  $\delta_{{\rm GNS}}[u] > 0$ unless $u$ is a multiple of $v_{\lambda,x_0}$ for some $\lambda>0$ and some $x_0\in \R^2$.  The question addressed in this paper is:  
 
 \medskip
 \noindent{$\bullet$}
 {\it When  $\delta_{{\rm GNS}}[u] = 0$ is small, in what sense must
  $u$ be close to some multiple of $v_{\lambda,x_0}$?}
  \medskip
  
As indicated above, it is natural to think of the GNS inequality as an inequality concerning densities $\rho$, and hence it is natural to think
of  $\delta_{{\rm GNS}}$ in this way too. However, associated to each $u$ there are {\em two} natural densities to consider: $\rho(x) = u^6(x)$ and
$\sigma(x) =  u^4(x)$.
Indeed, $u^6$ is the density with appears in the optimal transportation proof
(and, for scaling reasons, it is the ``natural'' quantity to control using the deficit),
while $u^4$ is the density which appears in the application
to the Keller-Segel equations. Therefore, our notation refers to $\delta_{{\rm GNS}}$ as a function of $u$.

Our first main result is:

\begin{thm}
\label{thm:stability 6A}
Let $u \in W^{1,2}(\R^2)$ be a non-negative function such that $\|u\|_6 = \|v\|_6$. Then there exist universal constants
$K_1,\delta_1>0$ such that, whenever $\delta_{{\rm GNS}}[u] \leq \delta_1$,
\begin{equation}
\label{eq:stability 6A}
\inf_{\lambda>0,x_0\in \R^2}  \|u^6- \lambda^{2}v^6_{\lambda,x_0}\|_1\leq  K_1\delta_{{\rm GNS}} [u]^{1/2}.
\end{equation}
\end{thm}

\begin{remark}\label{rmk:stability 6}
Actually, since  $\|\lambda^{2}v^6_{\lambda,x_0}\|_1 = \|v^6\|_1$ and
$\|u^6- \lambda^{2}v^6_{\lambda,x_0}\|_1 \leq \|u^6\|_1+\|v^6\|_1=\pi$ for all $\lambda>0$,
\eqref{eq:stability 6} holds with $K_1=\pi/\delta_1^{1/2}$
whenever $\delta_{{\rm GNS}}[u] \geq \delta_1$. So, up to enlarging $K_1$,
\eqref{eq:stability 6}
always holds without any restriction on $\delta_{{\rm GNS}}[u]$.
Moreover, also the sign restriction on $u$ is superfluous; see e.g. \cite{LL} or
\cite[Proof of Theorem 1.1]{CFMP} for a finer result.
However, since for the applications we have in mind
$u$ will always be nonnegative  and we are only interested in the regime when $\delta_{{\rm GNS}}[u]$ is small,
we have chosen to state the theorem
in this simple form.
\end{remark}

To obtain a similar result for the density $u^4(x)$, we need to require additional {\it a-priori} bounds. Essentially what we need is some sort
of bounds ensuring uniform integrability of the class of densities satisfying the bounds. For the PDE applications we have in mind, it is natural to
use {\em moment bounds} and {\em entropy bounds}.

Define
\begin{equation}\label{fo2}
N_p(u) = \int_{\R^2} |y|^p u^4(y)\dd y  \qquad{\rm and}\qquad S(u) =  \int_{\R^2} u^4 \log (u^4) \dd y\ .
\end{equation}

\begin{thm}\label{4thstab} Let $u \in W^{1,2}(\R^2)$ be a non-negative function such that $\|u\|_4 = \|v\|_4$. Suppose also that 
for some $A,B < \infty$ and some $1< p < 2$, 
\begin{equation}\label{condsZ}
S[u] =  \int_{\R^2} u^4 \log (u^4) \dd x  \leq A< \infty \qquad{\rm and}\qquad  N_p[u]  := \int_{\R^2} |y|^p u^4(y)\dd y\leq B< \infty\,.
\end{equation}
Moreover, assume that
\begin{equation}\label{condsZZ}
\int_{\R^2}y u^4\dd y = 0\ .
\end{equation}
Then, for constants $K_2,\delta_2>0$ depending only on $p$, $A$, and $B$, 
whenever $\delta_{{\rm GNS}}[u] \leq \delta_2$, 
\begin{equation}\label{infbndZ}
\inf_{\lambda>0} \|u^4 - \lambda^{2}v_\lambda ^4\|_1 \leq K_2 \delta[u]_{{\rm GNS}}^{(p-1)/(4p)}\ .
\end{equation}
Moreover, there is a constant $a>0$, depending only on $A$ and $B$, such that the infimum in (\ref{infbndZ}) is achieved at some $\lambda \in [a,1/a]$. 
\end{thm}

To explain how to prove these results,
let us first recall that a stability result for the sharp Sobolev inequality  (\ref{sharpsob}) has been proved some time ago  by Bianchi and Egnell \cite{BE}.
It states that
there is a constant $C_n$, $n\geq 3$, so that for all $f\in W^{1,2}(\R^n)$, 
\begin{equation}\label{bestab}
\|\nabla f\|_2^2 - S_n\|f\|_{2n/(n-2)}^2 \geq C_n \inf_{c,\mu>0 ,\ x_0\in \R^n}\|\nabla f - c\nabla h_{\mu,x_0}\|_2^2
\end{equation}
where
$$h_{\mu,x_0}(x) :=  (1+ \mu^2|x-x_0|^2)^{-(n-2)/2}\ .$$

The proof uses a compactness argument  so there is no information on the value of $C_n$. On the other hand, the metric used on the
right hand side in (\ref{bestab}) is as strong as one could hope for,
and in this sense
the result of Bianchi and Egnell is remarkably strong. 

Unfortunately, the fact that typical GNS inequalities involve three norms and not two,
prevents any direct adaptation of the proof of
Bianchi and Egnell to any of the other cases of the GNS inequality for which the optimizers are known. Moreover, other recent proofs for stability 
based on optimal transportation \cite{FiMP1} or symmetrization techniques
\cite{FMP1,FMP2,CFMP,FiMP2} did not produce (at least up to now) any results in this situation.

However, it has recently been shown \cite{BGL} that one may deduce the  {\em sharp forms} of the GNS inequalities in (\ref{gnsdd})
from the sharp Sobolev inequality (\ref{sharpsob}). Of course, it is quite easy to deduce the GNS inequalities with a non-optimal constant
from the Sobolev inequality  and H\"older's inequality. The argument in \cite{BGL}, which we learned from Dominique Bakry,
is more subtle: In particular, as we explain in the next section, one deduces the particular
{\em two-dimensional} GNS inequality  (\ref{ourgns})
from the  {\em four-dimensional} Sobolev inequality. 

This derivation of  (\ref{ourgns}) provides the beginnings of a bridge between the Bianchi-Egnell stability result for the Sobolev inequality
and our theorems on stability for (\ref{ourgns}). Building and crossing the bridge still requires further work, and this is carried out in
Section 2 of the paper where we prove Theorems~\ref{thm:stability 6A} and \ref{4thstab}. \\

The third section of the paper concerns  two evolution equations and  three functionals, all with close connection to the GNS inequality (\ref{ourgns}).  The two equations, both describing the evolution of mass densities on $\R^2$, are:

\medskip
\noindent{\it (1)}  A two dimensional fast diffusion equation:
\begin{equation}\label{fdA}
 \frac{\partial \sigma}{\partial t}(t,x) =
 \Delta \sqrt{\sigma(t,x)}  +
 2\sqrt{\frac{\pi}{\kappa M}} \,{\rm div}(x\,\sigma(t,x))\ .
 \end{equation}
 Here $\kappa$ and $M$ are positive parameters that set the scale and mass of stationary solutions,
 as we shall explain. (It will be convenient to keep them separate).

\medskip
\noindent{\it (2)}  The {\em Keller-Segel equation}:
\begin{equation}\label{KS}
\frac{\partial \rho}{\partial t}(t,x) ={\rm div } \bigl[ \nabla \rho(t,x)-\, \rho(t,x)\nabla c(t,x)\bigr]\ ,
\end{equation}
where
$$c(t,x)=-\frac{1}{2\,\pi}\int_{\R^2}\log|x-y| \rho(t,y) \dd y\,.
$$

The fast diffusion equation (\ref{fdA}) has the steady state solutions
\begin{equation}\label{sted}
 \sigma_{\kappa,M}(x):=\frac{M}{\pi}\frac{\kappa}{{\left(\kappa +|x|^2 \right)^2}}\,.
\end{equation}
Note that $\int_{\R^2}\sigma_{\kappa,M}(x)\dd x = M$ for all $\kappa$.

The densities  in (\ref{sted}) with $M = 8\pi$ are also the steady states of the Keller-Segel system \eqref{KS}, and $M=8\pi$ is the {\em critical mass} for \eqref{KS}:
If the initial data has a mass less than $8\pi$, diffusion dominates and the solution diffuses away to infinity; if the 
initial data has a mass greater than $8\pi$, the restoring drift dominates and the solution collapses in finite time \cite{DP}.

We now remark that each $\sigma_{\kappa,M}$ is the  fourth power of a GNS optimizer; equivalently, they are multiples of the densities $v_\lambda^4$ that figure in Theorem~\ref{4thstab}.  This is the first indication of a close connection of these two equations to one another and to the GNS inequality (\ref{ourgns}). 

To go further, we note that both of these equations are gradient flow for the $2$-Wasserstein metric in the sense of Otto \cite{otto}. (For this fact, and further background on the Wasserstein metric, gradient flow, and these equations, see \cite{BCC}.)

The fast diffusion equation is gradient flow for 
the functional $\H_{\kappa,M}$, where:

\begin{defi}[Fast diffusion entropy]\label{fdent}
 \begin{equation}\label{def Hlambda}
\H_{\kappa,M}[\sigma] := \int_{\R^2}\frac{|\sqrt{\sigma}(y) - \sqrt{\sigma_{\kappa,M}}(y)|^2}{ \sqrt{\sigma_{\kappa.M}}(y)}\dd y
\end{equation}
\end{defi}
It is evident that $\H_{\kappa,M}[\sigma] $ is uniquely minimized at $\sigma = \sigma_{\kappa,M}$, and it is very easy to deduce an $L^1$
 stability result for this functional; see \cite{BCC}. 

On the other hand, the Keller-Segel system is gradient flow for the following ``free energy'' functional:
\begin{equation}\label{fks}
\F_{{\rm KS}}[\rho]  = \int_{\R^2} \rho(x) \log \rho(x)\dd x  +
\frac{1}{4\pi} \iint_{\R^2 \times \R^2} \rho(x) \log |x-y| \rho(y)\dd x \dd y \ .
\end{equation}

We are concerned with the critical mass case $M = 8\pi$, in which case this coincides the the logarithmic Hardy-Littlewood-Sobolev (Log-HLS)
functional:

\begin{defi}[Log-HLS Functional]\label{lhlsfude}
 The Log-HLS functional $\F$  is defined by
$$
\F[\rho] := \,\int_{\R^2}\rho (x)\log \rho(x)\dd x
+2\left( \int_{\R^2}\rho(x)\dd x\right)^{-1}\!\!\!\!
\iint_{\R^2\times\R^2}\rho(x) \log|x-y| \rho(y)\dd x \dd y \nonumber\ $$
on the domain consisting of  densities $\rho$ on $\R^2$ such that both $\rho \ln \rho$ and $\rho\ln
(e+|x|^2)$ belong  to $L^1(\R^2)$ (we define $\F[\rho] := +\infty$ otherwise). 
\end{defi}

The logarithmic HLS functional $\F$
is invariant under scale changes: for $a>0$ and $\rho$ in
the domain of $\F$,  $\F[\rho] = \F[\rho_{(a)}]$ for all $a>0$, where $\rho_{(a)} := a^2\rho(ax)$.  In particular,
$\F[\sigma_{M,\kappa}]$ is independent of  $\kappa$. One computes \cite{Be,CarLos}
\begin{equation}\label{cmdef}
\F[\sigma_{M,\kappa}] := C(M)  =  M(1+\log\pi-\log(M))\ .
\end{equation}

The sharp Log-HLS
inequality \cite{Be,CarLos} states that $\F[\rho] \geq C(M)$ for all densities of mass $M>0$.
Moreover, there is equality if and only if 
$\rho(x) = \sigma_{\kappa,M}(x-x_0)$
 for  some $\kappa > 0$ and some $x_0\in
\R^2$.  Thus, among densities of fixed mass $M$, the $\sigma_{\kappa,M}$ are the unique minimizers of $\F$. However, in contrast with the fast diffusion entropy
$\H_{\kappa,M}$, is not so simple to deduce an
$L^1$ stability result for the Log-HLS inequality (i.e., for the minimization problem associated to $\F$ at fixed mass).  One of the main results 
proved in Section 3 is a stability result for this inequality; see Theorem~\ref{loghlsstab} below. 

The fact that the fast diffusion equation  (\ref{fdA}) is a 
gradient flow for $\H_{\kappa,M}$ implies that  $\H_{\kappa,M}[\sigma(\cdot,t)]$ is monotone decreasing
along solutions of  (\ref{fdA}) with initial data for which $\H_{\kappa,M}[\sigma(\cdot,0)]$ is finite.
Likewise, the fact that the Keller-Segel equation is a gradient flow for
the functional $\F_{{\rm KS}}$ implies that the Log-HLS functional $\F[\rho(\cdot,t)]$ is decreasing along solutions  of (\ref{KS}) for initial data with
the critical mass $M = 8\pi$  such that $\F[\rho(\cdot,0)]$ is finite. 

 There is, nonetheless, a fundamental difference: {\em The functional 
 $\H_{\kappa,M}$ is uniformly displacement convex \cite{BCC}}, and as shown by Otto \cite{otto}, evolution equations that are ${\rm W}_2$-gradient flows of 
uniformly displacement convex functionals have an exponential rate of convergence to equilibrium; i.e., the minimizers of the functional. 
 This yields an exponential rate of convergence to equilibrium for the fast diffusion equation.  
 
 However, the Log-HLS functional {\em is not} displacement convex (nor is it even convex in the uusal sense), and hence the gradient flow structure by itself does not provide any sort of
 rate of convergence for this equation. We shall show that our quantitative stability estimates for the GNS inequality lead to a stability result for the Log-HLS inequality,
 and combining these  results we get a quantitative 
 rate of convergence estimate for the Keller-Segel equation; see Theorem \ref{thm:KS}.

 A  key to this is a surprising interplay between $\H_{\kappa,M}$ and $\F$
 along our two evolutions. As noted above, by their nature as gradient flow evolutions, it is naturally true  that $\F[\rho(\cdot,t)]$
 decreases along critical mass solutions of the Keller-Segel equation, and it is naturally true that 
 $\H_{\kappa,M}[\rho(\cdot,t)]$ decreases along solutions of the fast diffusion equation  (\ref{fdA}).

More surprisingly, it has recently been shown \cite{CCL,BCC} that, in fact, $\F[\sigma(\cdot,t)]$  is also decreasing along  solutions of (\ref{fdA}),
and that 
$\H_{\kappa,8\pi}[\rho(\cdot,t)]$ is also decreasing along critical mass solutions of the Keller-Segel equation (\ref{KS}).  

In fact, as shown in \cite{CCL}, for any solution $\sigma(x,t)$ of (\ref{fdA}), 
\begin{equation}\label{disslem}
\frac{{\rm d}}{{\rm d}t} \F[\sigma(\cdot,t)] = - \frac{8\pi}{M}\D[ \sigma(\cdot,t)]\ ,
\end{equation}
where $\D$ denotes the dissipation functional defined as follows:  

\begin{defi}[Dissipation functional]\label{dissdef}
For any density $\sigma$ on $\R^2$, let $u := \sigma^{1/4}$.   If $u$ has a square integrable distributional gradient, define
\begin{equation}\label{diss}
{\mathcal D}[\sigma] :=  \frac{1}{\pi}\left(\|\nabla u\|_2^2\|u\|_4^4 - \pi \|u\|_6^6\right)\ .
\end{equation}
Otherwise, define ${\mathcal D}[\sigma]$ to be infinite.  
\end{defi}

Note that ${\mathcal D}[\sigma] \geq 0$ as a consequence (actually, a restatement) of the sharp GNS inequality (\ref{ourgns}).

What is actually proved
for the critical mass Keller-Segel equation is somewhat less. However, in
\cite{BCC}, a natural class of solutions called ``properly dissipative solutions'' is constructed, along which
\begin{equation}\label{entdec}
\H_{\kappa,8\pi}[\rho(\cdot,T)] + \int_0^T \D[\rho(\cdot,t)] \dd t  \leq \H_{\kappa,\pi}[\rho(\cdot,0)]\
\end{equation}
 for all $T>0$. (This is evidently an analog of  (\ref{disslem}) in integrated form.)\\

Our goal here is to understand the asymptotic behavior of a properly dissipative solution
$\rho(t)$ of \eqref{KS} starting from some $\rho$ such that $\H_{\kappa,\pi}[\rho] < \infty$.

The first observation is that, as an immediate consequence of \eqref{entdec}, for any $T>1$
\begin{equation}\label{decay}
\inf_{t\in [1,T]}\D[\rho(\cdot,t)] \leq \frac1{T-1}  \int_{1}^T \D[\rho(\cdot,t)] \dd t  \leq  \frac1{T-1}  \H_{\kappa,\pi}[\rho] \ .
\end{equation}
(As we will see in Section \ref{sect:KS},
the reason for considering $t \in [1,T]$ is to ensure that some time passes
so that the solution enjoys some further regularity properties
needed to apply our estimates.)

Now, observe that for any density $\sigma$ on $\R^2$, 
 \begin{eqnarray}\label{pass}
 {\mathcal D}[\sigma]  &=& (\|\nabla \sigma^{1/4}\|_2\|\sigma^{1/4}\|_4^2 + \sqrt{\pi} \|\sigma^{1/4}\|_6^3)\delta_{{\rm GNS}}(\sigma^{1/4})\nonumber\\
 &=& (\|\nabla \sigma^{1/4}\|_2\|\sigma\|_1^{1/2} + \sqrt{\pi} \|\sigma\|_{3/2}^2)\delta_{{\rm GNS}}(\sigma^{1/4})\ .
 \end{eqnarray}
Hence, granted (for now) an {\em a-priori}  bound on $\int_{\R^2}|x|^p \rho(x,t)\dd x$ for some $1< p < 2$
for $t \geq 1$, 
we have a lower bound on $\|\rho(\cdot,t )\|_{3/2}$ depending only on the $p$th 
moment bound.
From this and \eqref{decay} we deduce that, for any $T \geq 2$, there exists \textit{some}  $\bar t\in [1,T]$ such that
$$
\delta_{{\rm GNS}}[\rho^{1/4}(\cdot,t)] \leq   \frac{C}{T}   \H_{\kappa,\pi}[\rho] \ ,
$$
where $C$ is universal (as it depends only on the $p$th moment bound).

Then, granted also   an {a-priori} upper bound on the entropy $\int_{\R^2} \rho \log \rho(x,t)\dd x$ for $t \geq 1/2$, 
applying Theorem \ref{4thstab} we conclude that for \textit{some} $\mu>0$,
\begin{equation}\label{upbnd1}
\|\rho(\cdot,\bar t) - \sigma_{\mu,8\pi}\|_1  \leq
C \left(\frac{1}{T}   \H_{\kappa,\pi}[\rho] \right)^{(p-1)/4p}\ ,
\end{equation}
(recall that the density $v_\lambda^4$ is a multiple of some   $\sigma_{\mu,8\pi}$).

The inequality (\ref{upbnd1}) bounds the time it takes a solution of the critical mass Keller-Segel equation to approach
$\sigma_{\mu,8\pi}$ for  {\em some} $\mu$. 
However, to get a quantitative convergence result, we must do two more things: 
First, show that $\rho(\cdot,t)$
approaches $\sigma_{\mu,8\pi}$ for $\mu = \kappa$, and then show that eventually it {\em remains} close. 

The first point is relatively easy, since $\H_{\kappa,8\pi}[\sigma_{\mu,8\pi}] = \infty$ for $\mu\neq \kappa$ (because of the sensitivity of 
$\H_{\kappa,M}[\rho]$ to the tail of  $\rho$; see \cite{BCC}). 

The second requires more work: The strategy used in \cite{BCC} was to show that eventually 
$\F[\rho(\cdot,t)]$ becomes small. Since this quantity is monotone, once small, it stays small. 
Then one uses a stability inequality for the Log-HLS inequality to conclude that 
$\|\rho(\cdot,t) - \sigma_{\kappa,8\pi}\|_1$ stays small.   The stability inequality for the Log-HLS inequality used in
 \cite{BCC} relied on a compactness argument, and thus gave $L^1$ convergence to the steady state, but without any rate estimate. 
 Moreover, the argument in \cite{BCC} also used compactness arguments to deduce that 
 $\|\rho(\cdot,t) - \sigma_{\mu,8\pi}\|_1$ eventually becomes small for some $\mu$, so that there was no quantitative estimate
 on the time to first  approach the set of densities $\{\sigma_{\kappa,8\pi}\: \ \kappa>0\}$.
 
To provide a convergence result with quantitative bounds we do the following:
First we show almost Lipschitz regularity of $\mathcal F$ in $L^1$ (Theorem \ref{loghlsfunccont}),
and we combine it with (\ref{upbnd1}) and the fact that $p$ can be chosen close to $2$,
to deduce that
$$
\F[\rho(\cdot,\bar t)] -C(8\pi) \leq C T^{-(1-\epsilon)/8},
$$
where $C(M)$ is defined in (\ref{cmdef}). 
Since $\bar t \leq T$ and $\F[\rho(\cdot,t)]$ is decreasing, we deduce that
\begin{equation}\label{exfredec}
\F[\rho(\cdot,T)] -C(8\pi) \leq C T^{-(1-\epsilon)/8}
\end{equation}
\textit{for all} $T \geq 2$.

This brings us to our final stability result:

\begin{defi}[Log-HLS deficit]\label{loghlsdef}
For any density $\rho$ on $\R^2$ with $\int_{\R^2}\rho(x)\dd x = M$, 
we define $\delta_{\rm HLS}[\rho]$, the deficit in the log HLS inequality, as
\begin{equation}\label{app5}
\delta_{\rm HLS}[\rho] = \F[\rho] - M(1+\log\pi-\log(M))\ .
\end{equation}
\end{defi}

\begin{thm}[Stability for Log-HLS]\label{loghlsstab} 
Let $\rho$ be a density  of mass $M$ on $\R^2$ such that, for some $\kappa>0$,
\begin{equation}\label{hlbnd}
\H_{\kappa,M}[\rho] =: B_\H < \infty\ .
\end{equation}
Assume also that
\begin{equation}\label{fbnd}
\F[\rho] =: B_\F < \infty\ ,
\end{equation}
and
\begin{equation}\label{dlbnd}
\D[\rho] =: B_\D < \infty\ .
\end{equation}
Then,  for all $\epsilon>0$, there is a constant $C$, depending only on $\epsilon$, $M$, $\kappa$, $B_\H$, $B_\F$ and  $B_\D$,
such that
$$\| \rho-\sigma_{\mu,M} \|_1 \leq C \left( \delta_{\rm HLS}[\rho]\right)^{(1-\epsilon)/20} $$
for some $\mu>0$, provided $\delta_{\rm HLS}[\rho]$ is sufficiently small.
\end{thm}

By (\ref{exfredec}),  $\delta_{\rm HLS}[\rho(\cdot,t)]$ is decreasing to zero at a rate of (essentially) $t^{-1/8}$. 
Then, by Theorem \ref{loghlsstab}, there exists some $\mu(t)$ such that
$$\| \rho(\cdot,t)- \sigma_{\mu(t),8\pi}\|_1$$
converges to zero at a rate of essentially 
$t^{-1/160}$. 
Finally, a simple argument using the sensitive dependence of $\H_{\kappa,8\pi}$ on tails allows us to show that $\mu(t)$ converges at a logarithmic rate to $\kappa$.

It is interesting that the approach to equilibrium described by these quantitative bounds takes place on two separate time scales: The solution
 approaches the one-parameter family of (centered) stationary states with at least a polynomial rate.  Then, perhaps  much more gradually, 
 at only a logarithmic rate, the solution adjusts its spatial scale
to finally converge to the unique stationary solution within its basis of attraction.  It is reasonable to expect such behavior:  The initial data may, for example, be exactly equal to $\sigma_{\kappa,8\pi}$ on the complement of a ball of very large radius $R$, and yet may ``look much more like''
$\sigma_{\mu,8\pi}$ on a ball of smaller radius for some $\mu\neq \kappa$. One can then expect the solution to first approach 
$\sigma_{\mu,8\pi}$, and then only slowly begin to feel its distant tails and make the necessary adjustments to the spatial scale. 

The precise statement of our results on the rates of convergence for the critical mass Keller-Segel equation is given in  Theorem \ref{thm:KS} below.

We close this introduction by remarking that the key to the proof of  Theorem \ref{loghlsstab} is (\ref{disslem}), which, upon integration, yields
an expression for the Log-HLS deficit that can be related to the GNS deficit studied
in Section \ref{sect:GNS}.

\section{Stability results for GNS inequalities}
\label{sect:GNS}

In the forth-coming book \cite{BGL} the authors  present a very elegant argument
to deduce the  family of sharp Gagliardo-Nirenberg inequalities (\ref{gnsdd})
as a simple corollary of the sharp Sobolev inequality (\ref{sharpsob}). The argument has been known for some time
in certain circles, and is referred to as a result of D. Bakry in the third part of the remark following \cite[Theorem 4]{CNV}.
We are grateful to D.  Bakry  for communicating this proof to us, and for providing us with a draft of the relevant chapter of \cite{BGL}.

Here, starting from this proof and combining it with the quantitative stability result (\ref{bestab}) 
of Bianchi and Egnell, we deduce several 
stability results for the Gagliardo-Nirenberg-Sobolev  inequality (\ref{ourgns})  of that family.

Although much of the argument below could be carried out for this whole family (modulo
being able to extend the argument of Bianchi-Egnell to a slightly more general situation),
we prefer to consider only the  one particular GNS inequality which is important for the applications we consider here. 
In this way we also avoid the risk of making the paper excessively involved and hiding the main ideas.

\subsection{From Sobolev to GNS}
We begin by explaining the argument of \cite{BGL} specialized to our particular case of interest.

\if false
The following GNS inequality was been recently proved in \cite{DD}:
\begin{equation}\label{rob6b}
 \int_{\R^2} |\nabla u|^2\dd y\int_{\R^2} u^4 \dd y \geq 
\pi \int_{\R^2} u^6\dd y 
\end{equation}
with equality if and only if $u(x) = cv(\sqrt{\mu}(y-y_0))$
for some $c\in \R$, $\mu>0$ and $y_0\in \R^2$, where
\begin{equation}\label{def v}
v(y) := \frac{1}{\sqrt{1+|y|^2}}\ .
\end{equation}
In order to fix the parameters $c$ and $\mu$ to be both equal to $1$,
in the sequel we will work with functions $u \in W^{1,2}(\R^2)$ such that
\begin{equation}
\label{eq:hyp GN u}
\|u\|_6=\|v\|_6=\frac{\pi}{2}\ ,\qquad \sqrt{2}\|\nabla u\|_2 = \|u\|_4^2\ ,
\end{equation}
so that $v$ will be the only minimizer in this class.
\fi

The four-dimensional version of the sharp Sobolev inequality (\ref{sharpsob}) has the explicit form
\begin{equation}\label{s4d}
\|f\|_4^2 \leq \frac{1}{4\pi}\sqrt{\frac{3}{2}}\|\nabla f\|_2^2\ ,
\end{equation}
and equality holds if $f=g$, where 
\begin{equation}\label{rob1}
g(x,y) :=  \frac{1}{1+ |y|^2+|x|^2}\,\qquad x,y \in \R^2\  .
\end{equation}
The key observation, which is at the core of the proof of the next result,
is that $g$ can be written as
$$
g(x,y) =   \frac{1}{G(y)+|x|^2}\,\qquad \text{with}\quad G(y) := v^{-2}(y) =   1+|y|^2.
$$
The following result, which is a particular case of the results in \cite[Chapter 7]{BGL}, relates \eqref{ourgns} and \eqref{s4d}.
\begin{prop}
\label{prop:Sob GN}
Let $u \in W^{1,2}(\R^2)$ be a non-negative function satisfying 
\begin{equation}
\label{eq:hyp GN u}
\|u\|_6=\|v\|_6=\frac{\pi}{2}\ ,\qquad \sqrt{2}\|\nabla u\|_2 = \|u\|_4^2\ ,
\end{equation}
 and define
$f:\R^4 \to \R$ as
$$
f(x,y):=\frac{1}{F(y)+|x|^2},\qquad F(y) := u^{-2}(y), \qquad x,y \in \R^2.
$$
Then
\begin{equation}\label{robG1}
\delta_{{\rm GNS}}[u] = \left( \int_{\R^2} |\nabla u|^2\dd y\right)^{1/2}\left(\int_{\R^2} u^4 \dd y \right)^{1/2} - 
\left(\pi \int_{\R^2} u^6\dd y \right)^{1/2} =
\sqrt{3}\left( \frac{1}{4\pi}\sqrt{\frac{3}{2}}\|\nabla f\|_2^2 - \|f\|_4^2\right)\,.
\end{equation}
\end{prop}

\begin{remark}
Observe that, given $u \in W^{1,2}(\R^2)$ with $u \not\equiv 0$, we can always multiply it by a constant 
so that $\|u\|_6=\|v\|_6$, and then scale it as
$\mu^{1/3}u(\mu y)$ choosing  $\mu$ to ensure that
$\sqrt{2}\|\nabla u\|_2 = \|u\|_4^2$. Since \eqref{ourgns}
is invariant under this scaling, this proves (\ref{ourgns}). This is the use of the identity (\ref{robG1}) made in \cite{BGL}.
Our interest in this proposition is that it relates the GNS deficit to the Sobolev deficit. 
\end{remark}

\begin{proof}
We compute
\begin{eqnarray}
\|\nabla f\|_2^2 &=& \int_{\R^2}\left(  \int_{\R^2}\frac{|\nabla F(y)|^2}{(F(y)+|x|^2)^4}\dd x\right)\dd y + 
\int_{\R^2}\left(  \int_{\R^2}\frac{4|x|^2}{(F(y)+|x|^2)^4}\dd x\right)\dd y \nonumber\\
&=&\frac{\pi}{3} \int_{\R^2}|\nabla F(y)|^2 F^{-3}(y)\dd y +  
\frac{2\pi}{3} \int_{\R^2} F^{-2}(y)\dd y\ \nonumber
\end{eqnarray}
and
\begin{equation}
\label{eq:norm f 4}
\|f\|_4^2  =\left(\frac{\pi}{3} \int_{\R^2} F^{-3}(y)\dd y \right)^{1/2}\ .
\end{equation}
Thus
$$0 \leq \frac{1}{4\pi}\sqrt{\frac{3}{2}}\|\nabla f\|_2^2 - \|f\|_4^2 = \frac{1}{2\sqrt{6}}\left(2\int_{\R^2} |\nabla u|^2\dd y + \int_{\R^2} u^4 \dd y \right) - 
\left(\frac{\pi}{3} \int_{\R^2} u^6\dd y \right)^{1/2}\ ,$$
or equivalently (using the identity $2\sqrt{AB} = A+B - (\sqrt{A} - \sqrt{B})^2$)
\begin{eqnarray*}
\left( \int_{\R^2} |\nabla u|^2\dd y\right)^{1/2}\left(\int_{\R^2} u^4 \dd y \right)^{1/2} - 
\left(\pi \int_{\R^2} u^6\dd y \right)^{1/2} &=& 
\sqrt{3}\left( \frac{1}{4\pi}\sqrt{\frac{3}{2}}\|\nabla f\|_2^2 - \|f\|_4^2\right) \\
 &-& \frac{1}{2\sqrt{2}}\left(\sqrt{2}\|\nabla u\|_2 - \|u\|_4^2\right)^2\,.
 \end{eqnarray*}
Recalling that  $\sqrt{2}\|\nabla u\|_2 = \|u\|_4^2$  by assumption, and recalling the definition (\ref{rob6A}) of the GNS deficit,
the proof is complete.
\end{proof}

\subsection{Controlling the infimum in the Bianchi-Egnell Theorem.}
The family of functions
$$
g_{c,\mu,x_0,y_0}(x,y):= \frac{c\mu}{1+\mu^2|x+x_0|^2+\mu^2|y+y_0|^2 }, \qquad c \in \R,\, \mu >0, \, x_0,y_0 \in \R^2.
$$
consists of all of the optimizers of the Sobolev inequality (\ref{s4d}).
Observe that, with this definition, $g=g_{1,1,0,0}$, where $g$ is the function defined in (\ref{rob1}). 

The Bianchi-Egnell stability result  \cite{BE} combined with the Sobolev inequality \eqref{s4d}
asserts the existence of a universal constant $C_0$ such that 
\begin{equation}\label{beex}
C_0\sqrt{3}\left( \frac{1}{4\pi}\sqrt{\frac{3}{2}}\|\nabla f\|_2^2 - \|f\|_4^2\right)  \geq
\inf_{c,\mu,x_0,y_0}\|f -  g_{c,\mu,x_0,y_0}\|_4^2\ .
\end{equation}
Hence, whenever $u$ satisfies the conditions (\ref{eq:hyp GN u}) of Proposition~\ref{prop:Sob GN}, 
\begin{equation}
\label{delta u f g}
C_0 \delta_{{\rm GNS}}[u] \geq \inf_{c,\mu,x_0,y_0}\|f -  g_{c,\mu,x_0,y_0}\|_4^2\ .
\end{equation}
Let us observe that the renormalization $\|u\|_6=\|v\|_6$ is equivalent to $\|f\|_4=\|g\|_4$.

Our main goal in this subsection is to show first that, up to enlarging the constant $C_0$, we can assume that $c=\mu=1$ and $x_0=0$
(see Lemma \ref{lem: c mu x0} below).  This paves the way for the estimation of  the infimum on the right hand side of (\ref{delta u f g})
in terms of $u$ and $v$.

\if false
Then we will prove that $\|f- g_{1,1,0,y_0}\|_4$
controls $\|u^6- v^6 (\cdot -y_0)\|_1$ (Lemma \ref{lem: fg uv} and Theorem
\ref{thm:stability 6}), and finally,
assuming some controls on the moment, we will be able to kill the translation $y_0$
(Proposition \ref{prop:translation}).
\fi

 \begin{lm}
\label{lem: c mu x0}
Let $f$ be given by
$f(x,y)=1/(F(y)+|x|^2)$, with $F : \R^2 \to \R$ non-negative, and 
 $g$ be given by (\ref{rob1}).
Suppose that $\|f\|_4 = \|g\|_4$. Then there is a universal constant $C_1$ so that, for all real
numbers $\delta>0$ with
\begin{equation}\label{robu45}
\delta^{1/2} \leq \frac{1}{2400}\ ,
\end{equation}
 whenever
$$\|f -  g_{c,\mu,x_0,y_0}\|_4  \leq \delta^{1/2} \qquad{\rm for\ some}\qquad c, \mu, x_0, y_0\ ,$$
then
$$
\|f- g_{1,1,0,y_0}\|_4 \leq C_1\delta^{1/2}\ .
$$
\end{lm}
As can be seen from the proof, a possible choice for $C_1$ is 
$4800$.

\begin{proof} 
Suppose that   $\|f -  g_{c,\mu,x_0,y_0}\|_4 < \delta^{1/2}$ for some $\delta^{1/2}>0$ satisfying \eqref{robu45}.

\textit{$\bullet$ Step 1: we can assume $c=1$.}
First of all notice that $c \geq 0$, as otherwise
$$
\delta^2 \geq \int_{\R^4}|f-g_{c,\mu,x_0,y_0}|^4 \dd x \dd y\geq
\int_{\R^4}\left(|f|^4+|g_{c,\mu,x_0,y_0}|^4\right)\dd x \dd y \geq \|f\|_4^4=\|g\|_4^4=\frac{\pi^2}{6}\,,
$$ 
which is in contradiction with \eqref{robu45}.

Now, for any $c,\mu>0$ and $x_0,y_0 \in \R^2$,
$\|g_{c,\mu,x_0,y_0}\|_4 = c\|g\|_4 = c\|f\|_4$.  Hence, 
$$|c-1| \|g\|_4= |\|g_{c,\mu,x_0,y_0}\|_4 - \|f\|_4| \leq  \|f -  g_{c,\mu,x_0,y_0}\|_4 < \delta^{1/2}\ ,$$
and by the triangle inequality we get
\begin{eqnarray}
\label{eq:c 1}
\|f -  g_{1,\mu,x_0,y_0}\|_4 &\leq&  \|f -  g_{c,\mu,x_0,y_0}\|_4 +  \|g_{1,\mu,x_0,y_0}-  g_{c,\mu,x_0,y_0}\|_4\nonumber\\
&=&  \|f -  g_{c,\mu,x_0,y_0}\|_4+|c-1| \|g\|_4\nonumber\\
&\leq& 2\,\delta^{1/2}\ .
\end{eqnarray}
Thus, up to enlarging the constant, we may replace $c$ by $1$.

\textit{$\bullet$ Step 2: we can assume $x_0=0$.}
Observe that, by construction, $f$ is even in $x$. Therefore \eqref{eq:c 1} implies
$$ \|f -  g_{1,\mu,x_0,y_0}\|_4 = \|f -  g_{1,\mu,-x_0,y_0}\|_4 \leq 2\,\delta^{1/2} \ ,$$
and by the triangle inequality,
$$ \|g_{1,\mu,2x_0,y_0} -g_{1,\mu,0,y_0} \|_4 =  \|g_{1,\mu,x_0,y_0} -g_{1,\mu,-x_0,y_0} \|_4 \leq  4\,\delta^{1/2}\ .$$
However, a simple argument using the unimodality and symmetry properties of $ g = g_{1,1,0,0}$ shows that
$$a \mapsto  \|g_{1,\mu,ax_0,y_0} -g_{1,\mu,0,y_0} \|_4$$
is increasing in $a>0$, thus
$$
\|g_{1,\mu,x_0,y_0} -g_{1,\mu,0,y_0} \|_4 \leq \|g_{1,\mu,2x_0,y_0} -g_{1,\mu,0,y_0} \|_4\leq  4\,\delta^{1/2}.
$$
One more use of the triangle inequality gives
 \begin{equation}\label{robu3}
 \|f - g_{1,\mu,0,y_0} \|_4 \leq 6\,\delta^{1/2}\ .
 \end{equation}
Hence, up to further enlarging the constant, we may replace $x_0$ by $0$. 

\textit{$\bullet$ Step 3: we can assume $\mu=1$.}
Making a change of scale, we can rewrite (\ref{robu3})  as
 \begin{equation}\label{eq:rescaled}
\left\Vert   \frac{1}{\mu} \frac{1}{F(y/\mu)+|x|^2/\mu^2} - \frac{1}{1+|y-\mu y_0|^2+|x|^2} \right\Vert _{4} \leq 6\,\delta^{1/2}
 \end{equation}
Let $A:=\{(x,y) \in \R^4\ : \ |x| \leq 1,\ |y-\mu y_0| \leq 1\}$.
Note that the Lebesgue measure of $A$ is $ \pi^2$.
Moreover, by a simple Fubini argument,
for any set $B\subset A$ with measure greater than $(15/16)\pi^2$ there exists
$\bar y \in \{y\ :\ |y-\mu y_0| \leq 1\}$ such that the set $B \cap \left(\R^2\times \{\bar y\}\right)$
must intersect both
$$A\cap \{(x,y)\ :\ |x| < 1/4\}\quad{\rm and}\quad 
A\cap \{(x,y)\ :\ |x| > 3/4\}\ .$$
(Indeed, if this was not the case, by Fubini Theorem the measure of $B$ would be smaller than $(15/16)\pi^2$.)

Now, applying Chebyshev's inequality, by \eqref{eq:rescaled}
we get the existence of a set $B \subset A$ of  measure at least $(31/32)\pi^2$
such that
\begin{equation}
\label{eq:ptwise bound}
\left| \frac{1}{\mu} \frac{1}{F(y/\mu)+|x|^2/\mu^2} - \frac{1}{1+|y-\mu y_0|^2+|x|^2}  \right|\leq
12\,\delta^{1/2} \qquad \forall\, (x,y) \in B
\end{equation}
(as the complement of the above set has measure less or equal than $1/16$, which is less than $\pi^2/32$).

Set $\alpha:=1+|y- \mu y_0|^2+|x|^2$ and $\b:=\mu \left(F(y/\mu)+|x|^2/\mu^2\right)$, so that
\eqref{eq:ptwise bound} becomes
\begin{equation}
\label{eq:ptwise bound 2}
\left| \frac{1}{\beta} - \frac{1}{\alpha}\right| \leq 12 \, \delta^{1/2} \qquad \forall\, (x,y) \in B.
\end{equation}
We observe that $\a \leq 3$ on $B$. Moreover, thanks to \eqref{robu45}, 
$$\frac{1}{\beta} \geq \frac{1}{\alpha} - \left| \frac{1}{\beta} - \frac{1}{\alpha}\right|
\geq \frac{1}{3} - 12 \, \delta^{1/2} \geq \frac{1}{4}\qquad \text{inside }B\ ,$$
that is $\b \leq 4$ on $B$.
Hence \eqref{eq:ptwise bound 2} gives
$$
|\a-\b| \leq 12\a \b\, \delta^{1/2}    \leq 144\,\delta^{1/2}\qquad \text{inside }B\ ,
$$
or equivalenly
$$
\left|1+|y-\mu y_0|^2+|x|^2 - \mu F(y/\mu)-|x|^2/\mu \right| \leq 144\,\delta^{1/2}
\qquad \forall\, (x,y) \in B.
$$
By the observation above,
we have chosen $B$ large enough that there exists $\bar y \in \{y \ : |y-\mu y_0| \leq 1\}$
such that $(x_1,\bar y),(x_2,\bar y)\in B$, with $|x_1| \in [0,1/4]$ and $|x_2| \in [3/4,1]$.
Then the above estimate gives
\begin{eqnarray*}
\frac{1}{2}\left|1 -\frac{1}\mu\right|&\leq& \left(|x_2|^2 - |x_1|^2 \right) \left|1 -\frac{1}\mu\right|\\
&\leq &  \left|1+|\bar y- \mu y_0|^2+|x_1|^2 - \mu F(\bar y/\mu)-|x_1|^2/\mu \right|\\
&&+\left|1+|\bar y-\mu y_0|^2+|x_2|^2 - \mu F(\bar y/\mu)-|x_2|^2/\mu \right|\\
&\leq & 288\, \delta^{1/2} \leq 300 \, \delta^{1/2}\, .
\end{eqnarray*}
Using \eqref{robu45} and the identity
$(\mu -1)(1-(1-1/\mu)) = (1-1/\mu)$, we easily deduce
\begin{equation}
\label{eq:bound mu}
|\mu-1| \leq 1200\delta^{1/2}\ .
\end{equation}
Since (as it is easy to check by a direct computation)
$$
|\partial_\mu g_{1,\mu,0,y_0}(x,y)| \leq \frac{|g_{1,\mu,0,y_0}(x,y)|}{\mu} \leq 2 |g_{1,\mu,0,y_0}(x,y)| \qquad \forall \,\mu\in [1/2,2]\,
$$
we get
$$
\|g_{1,\mu,0,y_0} - g_{1,1,0,y_0}\|_4 \leq 2|\mu-1|\|g\|_4 \qquad \forall \,\mu \in [1/2,2].
$$
Combining this with (\ref{robu3}) and \eqref{eq:bound mu}, we finally obtain
 \begin{equation}\label{robu35}
 \|f - g_{1,1,0,y_0} \|_4 \leq \left(6 + 2400\|g\|_4\right) \delta^{1/2} \leq 4800\, \delta^{1/2}\ ,
 \end{equation}
concluding the proof. 
\end{proof}

\subsection{Bounding $\|u^6- v^6\|_1$ in terms of $\|f - g\|_4$.}

Our goal in this subsection is  to bound $\|u^6-v^6(\cdot - y_0)\|_1$ from above in terms of  $\|f - g_{1,1,0,y_0}\|_4$.
\begin{lm}
\label{lem: fg uv}
Let $u \in W^{1,2}(\R^2)$ be a non-negative function satisfying \eqref{eq:hyp GN u}, and let $f$ be defined as in Proposition \ref{prop:Sob GN}.
Suppose that $\|f - g_{1,1,0,y_0}\|_4 \leq 1$. Then
$$\|u^6 - v^6(\cdot - y_0)\|_1 \leq  C_2 \|f-g_{1,1,0,y_0}\|_4\ $$
for some universal constant $C_2$.
\end{lm}

As can be seen for the proof, a possible choice for $C_2$ is $1000$.
We also remark that, by considering  $u$ of the form $v+\varepsilon \phi$ with $\varepsilon>0$ small, one sees that
the unit in the  above estimate is optimal.

\begin{proof}
Up to replace $u$ and $f$ by $u(\cdot - y_0)$ and $f(\cdot - y_0)$ respectively, we can assume that $y_0=0$.

We write
\begin{eqnarray}
\|f - g\|_4^4 &= &\int_{\R^2}\left(\int_{\R^2} \frac{|F- G|^4}{(F+|x|^2)^4(G+|x|^2)^4} \dd x\right) \dd y\nonumber\\
&\geq& \int_{\{F<G\}}\left(\int_{\R^2} \frac{|F- G|^4}{(F+|x|^2)^4(G+|x|^2)^4} \dd x\right) \dd y\nonumber\\
&&+\int_{\{F>G\}}\left(\int_{\R^2} \frac{|F- G|^4}{(F+|x|^2)^4(G+|x|^2)^4} \dd x\right) \dd y\nonumber.
\end{eqnarray}
By symmetry, it suffices to  estimate the first integral in the last expression.
We split $\{F<G\}=\{G/2\leq F<G\}\cup \{F<G/2\}=:A_1 \cup A_2$.

On $A_1$ we compute
\begin{eqnarray*}
\int_{A_1}\left(\int_{\R^2} \frac{|F- G|^4}{(F+|x|^2)^4(G+|x|^2)^4} \dd x\right) \dd y
&\geq& \int_{A_1}\left(\int_{\R^2}\frac{|F- G|^4}{(G+|x|^2)^8}\dd x\right) \dd y\\
&=& \frac{\pi}{7} \int_{A_1}\frac{|F- G|^4}{G^7} \dd y.
\end{eqnarray*}
Now, since $1/G^3=u^6 \in L^1$,  and since $\|u\|_6 = \|v\|_6=\pi/2 \leq 2$, H\"older's inequality yields
$$
 \int_{A_1} \frac{|F- G|}{G^4} \dd y \leq  2^{3/4} \left(\int_{A_1}\frac{|F- G|^4}{G^7} \dd y\right)^{1/4}\ .
$$
Also, pointwise on $A_1$, 
$$|u^6- v^6|=\left|\frac{1}{F^3} - \frac{1}{G^3}\right| = \left| \frac{(G-F)(G^2+GF +F^2)}{F^3G^3}\right|
\leq 14\,\frac{|G-F|}{G^4}\ .$$
Combining the last three estimates, we have
\begin{equation}\label{robu73}
\int_{A_1}|u^6- v^6|\dd y \leq 28 \left(\frac{7}{\pi}\int_{A_1}\left(\int_{\R^2} \frac{|F- G|^4}{(F+|x|^2)^4(G+|x|^2)^4} \dd x\right) \dd y\right)^{1/4}
\end{equation}

For $A_2$, we observe that
\begin{eqnarray*}
\int_{A_2}\left(\int_{\R^2} \frac{|F- G|^4}{(F+|x|^2)^4(G+|x|^2)^4} \dd x\right) \dd y
&\geq& \int_{A_2}\left(\int_{\R^2} \frac{|F- G|^4}{(F+|x|^2)^4(G+|x|^2)^4} \dd x\right) \dd y\\
&\geq& \int_{A_2}\frac{|F- G|^4}{G^7}\left(\int_{\R^2}\frac{1}{(1+|x|^2)^4(F/G +|x|^2)^4}\dd x\right) \dd y\\
&\geq& \frac{1}{16}\int_{A_2}\frac{|F- G|^4}{G^7}\left(\int_{B_1}\frac{1}{(F/G +|x|^2)^4}\dd x\right) \dd y\\
&=&\frac{\pi}{48}\int_{A_2}\frac{|F- G|^4}{G^7} \left[
\frac{1}{(F/G)^3} - \frac{1}{1+(F/G)^3} \right]\dd y\\
&\geq&\frac{\pi}{92}\int_{A_2}\frac{|G-F|^4}{F^3G^4} \dd y,
\end{eqnarray*}
where we used that $(1+|x|^2)^4 \leq 16$ on $B_1$, and that $\frac{1}{2(F/G)^3} \geq \frac{1}{1+(F/G)^3}$.
Since $G/2>F$ on $A_2$,
${\displaystyle \frac{|G-F|^4}{F^3G^4} \geq \frac{1}{16}\frac{1}{F^3}}$ on $A_2$. Therefore
(using that $16\cdot 96 \leq 500 \pi$)
\begin{eqnarray}
\int_{A_2}\left(\int_{\R^2} \frac{|F- G|^4}{(F+|x|^2)^4(G+|x|^2)^4} \dd x\right) \dd y &\geq& \frac{1}{500}
\int_{A_2}\frac{1}{F^3} \dd y\nonumber\\
&\geq& \frac{1}{500}\int_{A_2}\left(\frac{1}{F^3} -\frac{1}{G^3}\right)\dd y
\nonumber\\
&=& \frac{1}{500}\int_{A_2} |u^6-v^6| \dd y.\nonumber
\end{eqnarray}

When $\|f-g\|_4 \leq 1$, the left hand side is not greater than $1$, and hence, taking the fourth root on the left,
we obtain
$$\left(\int_{A_2}\left(\int_{\R^2}
\frac{|F- G|^4}{(F+|x|^2)^4(G+|x|^2)^4} \dd x\right) \dd y\right)^{1/4} \geq 
\frac{1}{500}\int_{A_2} |u^6-v^6| \dd y\ $$
Combining this with (\ref{robu73}), we have
$$\|f - g\|_4^4 
\geq \left(\int_{\{F<G\}}\left(\int_{\R^2} \frac{|F- G|^4}{(F+|x|^2)^4(G+|x|^2)^4} \dd x\right)\right)^{1/4}
\dd y \geq  \frac{1}{500} \int_{\{u>v\}}(u^6- v^6)\dd y\ .$$
By symmetry we also get
$$
\|f - g\|_4^4  \geq  \frac{1}{500} \int_{\{u<v\}}(u^6- v^6)\dd y\ ,
$$
which concludes the proof.
\end{proof}

\subsection{Proof of Theorem~\ref{thm:stability 6A}}
First, suppose that
$u \in W^{1,2}(\R^2)$ is a non-negative function satisfying \eqref{eq:hyp GN u}. 
Collecting together \eqref{delta u f g} and Lemmas \ref{lem: c mu x0} and \ref{lem: fg uv},
we deduce that there exist universal constants
$K_1,\delta_1>0$ such that, whenever $\delta_{{\rm GNS}}[u] \leq \delta_1$, 
\begin{equation}
\label{eq:stability 6}
  \|u^6- v^6(\cdot -y_0)\|_1\leq  K_1\delta_{{\rm GNS}}[u]^{1/2}.
\end{equation}

Next,  $\delta_{{\rm GNS}}[u]$ and $\|u\|_6$ are both  unchanged if $u(y)$ is replaced by $u_\mu := \mu^{1/3}u(\mu y)$.
Thus, assuming only that $\|u\|_6 = \|v\|_6$, we may choose a scale parameter $\mu$ so that 
$\sqrt{2}\|\nabla u_\mu\|_2 = \|u_\mu\|_4^2$. We then learn that
$$\int_{\R^2}\left| \mu^2u^6(\mu y) - v(y- y_0)\right| \dd y  \leq  K_1\delta_{{\rm GNS}}[u]^{1/2}\ .$$
Changing variables once more, and taking $\lambda := 1/\mu$, we obtain
$$\int_{\R^2}\left| u^6( y) - \lambda^{2}v(\lambda y- y_0)\right| \dd y \leq K_1\delta_{{\rm GNS}}[u]^{1/2}\ ,$$
which proves (\ref{eq:stability 6A}) and concludes the proof.

\subsection{Controlling the translation}
So far we know that if $u$ satisfies \eqref{eq:hyp GN u}
there is {\em some translate} $\widetilde u(y) = u(y-y_0)$ of $u$ such that
\begin{equation}
\label{eq:tilde C}
\|\widetilde u^6- v^6\|_1\leq K_1\delta_{{\rm GNS}}[u]^{1/2} 
\end{equation}
for some universal constant $K_1$ (see Theorem \ref{thm:stability 6A}
and Remark \ref{rmk:stability 6}).

 Our goal in this section is to show that under the additional hypotheses that
  \begin{equation}\label{mo1}
  M_p[u] := \int_{\R^2}|y|^p u^6(y)\dd y < \infty \qquad \text{for some $p>1$}
   \end{equation}
  and
 \begin{equation}\label{mo2}
 \int_{\R^2}yu^6(y)\dd y = 0\ ,
 \end{equation}
 then  $\|u^6 - v^6\|_1$ will be bounded by a multiple of some fractional power of $\delta_{{\rm GNS}}[u]^{1/2}$, with the fractional power
 depending on how large $p$ can be taken in (\ref{mo1}). The power would still be $\delta_{{\rm GNS}}[u]^{1/2}$ if we could take $p=\infty$. However,
since $M_4(v)= + \infty$, the useful values of $p$ are those in the range $1< p < 4$. 
 
 \begin{prop}
\label{prop:translation}
Let $u \in W^{1,2}(\R^2)$ be a non-negative function satisfying  \eqref{eq:hyp GN u},
and suppose that $M_p(u) < \infty$ for some $1<p<4$, and that (\ref{mo2}) is satisfied.
 Then there are  constants $\tilde K,\tilde \delta>0$, with $\tilde K$ depending only on 
 $p$ and $M_p(u)$, and $\tilde \delta$ depending only on $p$, so that whenever $\delta_{{\rm GNS}}[u] \leq \tilde\delta$,
 $$\|u^6 - v^6\|_1 \leq \tilde K \delta_{{\rm GNS}}[u]^{(p-1)/2p}\ .$$
 \end{prop}
 
\begin{proof}
First note that
 $$y_0\|v\|_6^6 = \int_{\R^2}y \widetilde u^6(y)\dd y   =   \int_{\R^2}y  v^6(y)\dd y +  \int_{\R^2}y [\widetilde u^6(y)- v^6(y)]\dd y\ .$$
 By the symmetry of $v$, the first term on the right is zero. By H\"older's inequality,
 \begin{eqnarray}
 \int_{\R^2}|y| |\widetilde u^6(y)- v^6(y)| \dd y  &=&  \int_{\R^2} |y| |\widetilde u^6(y)- v^6(y)|^{1/p}|\widetilde u^6(y)- v^6(y)|^{1/q}\dd y\nonumber\\
 &\leq& \left(\int_{\R^2} |y|^p |\widetilde u^6(y)- v^6(y)|\dd y\right)^{1/p}\|\widetilde u^6- v^6\|_1^{1/q}\nonumber\\
  &\leq& \left(M_p(\widetilde u)^{1/p} + M_p(v)^{1/p}\right)\|\widetilde u^6- v^6\|_1^{1/q}\ ,\nonumber
  \end{eqnarray}
where $q=p/(p-1)$.
 Next we note that
 \begin{eqnarray}
 M_p(\widetilde u)^{1/p} &=& \left(\int_{\R^2}|y|^p\widetilde u^6(y)\dd y\right)^{1/p}\nonumber\\
 &=&  \left(\int_{\R^2}|y+y_0|^p u^6(y)\dd y\right)^{1/p}
 \leq M_p(u)^{1/p} + |y_0|\|v\|_6^{6/p} \nonumber
 \end{eqnarray}
 Combining the last three estimates, we obtain
 $$|y_0|\|v^6\|_1  \leq \left[M_p(u)^{1/p} + M_p(v)^{1/p} + |y_0|\|v\|_6^{6/p} \right] \|\widetilde u^6- v^6\|_1^{1/q}\ ,$$
 and therefore by \eqref{eq:tilde C}
 $$|y_0| \leq
\frac{\left(M_p(u)^{1/p}+ M_p(v)^{1/p} \right)K_1\delta^{1/2q}_{{\rm GNS}}[u]}
{ \|v^6\|_1 - \|v^6\|_1^{1/p}K_1\delta^{1/2q}_{{\rm GNS}}[u]}\ .$$
 
Hence, there exist a constant $K'$ depending only on $p$ and $M_p(u)$, and a constant $\delta'$ depending only on $p$,
 such that whenever $\delta_{{\rm GNS}}[u] < \delta'$, 
\begin{equation}\label{mo7}
|y_0| \leq K'\delta^{1/2q}(u)\ .
\end{equation}
Now note that, by the Fundamental Theorem of Calculus and H\"older's inequality,
$$\|\widetilde u^6 - u^6\|_1 \leq 6|y_0|\|\nabla u\|_2\|u\|_{10}^5\ .$$
Then by the GNS inequality
$\|u\|_{10} \leq C\|\nabla u\|_2^{2/5}\|u\|_6^{3/5}$
we obtain
\begin{equation}\label{mo5}
\|\widetilde u^6 - u^6\|_1 \leq  6C |y_0|\|\nabla u\|_2^3 \|u\|_6^3\ .
\end{equation}
We now want to control the right hand side.
By hypothesis,
$\sqrt{2}\|\nabla u\|_2 = \|u\|_4^2$ and
$\|\nabla u\|_2\|u\|_4^2 = \sqrt{\pi}\|u\|_6^3 + \delta_{{\rm GNS}}[u]$. Therefore
$$\|\nabla u\|_2^3 = \sqrt{\frac{\pi}{2}}\|u\|_6^3 + \frac{\delta_{{\rm GNS}}[u]}{\sqrt{2}}\ .$$
Again using the fact that $\|u\|_6 = \|v\|_6=\pi/2$ and that $\delta_{{\rm GNS}}[u]$ is small
(so in particular we can assume $\delta_{{\rm GNS}}[u] \leq 1$)
from (\ref{mo5}) we obtain that 
$$
\|\widetilde u^6 - u^6\|_1 \leq K''|y_0| \ ,
$$
for some universal constant $K''$.
Combining this with (\ref{mo7})  yields the result.
\end{proof}

\subsection{Bounding $\inf_{\lambda>0}\|u^4- v_\lambda^4\|_1$}

As noted in the introduction, in certain PDE applications of  stability estimate for the GNS inequality (\ref{ourgns}), $u^4$ will play the role of a mass density,
and it will be of interest to control  $\inf_{\lambda>0}\|u^4- v_\lambda^4\|_1$ assuming that some moments of $u^4$ exist,
and that the ``normalization'' assumptions
\begin{equation}\label{fo1}
\int_{\R^2} u^4(y)\dd y =   \int_{\R^2} v^4(y)\dd y \qquad{\rm and}\qquad
\int_{\R^2}y u^4(y)\dd y =0\  
\end{equation}
hold.  Since the GNS deficit functional is scale invariant, we cannot hope to get information on the minimizing value of $\lambda$
 out of a bound on $\delta_{{\rm GNS}}[u]$ alone. However, knowing that the density $u^4$ is close to $v^4_\lambda$ for {\em some} $\lambda$
 is a strong information that can be combined with other ones, specific to a particular application, that then fix the scale $\lambda$.
 We shall see an example of this in the next section.  Here we concentrate on proving Theorem~\ref{4thstab} which bounds $\inf_{\lambda>0}\|u^4- v_\lambda^4\|_1$
 in terms of $\delta_{{\rm GNS}}[u]$. 
 
We recall that Theorem~\ref{4thstab} refers to non-negative functions $u\in W^{1.2}(\R^2)$ that satisfy (\ref{fo1}) and also
certain moment and entropy conditions:  Recall we have defined
$$
N_p(u) = \int_{\R^2} |y|^p u^4(y)\dd y  \qquad{\rm and}\qquad S(u) = \int_{\R^2} u^4 \log (u^4) \dd y\ .
$$
and  Theorem~\ref{4thstab} also requires that for some $A,B < \infty$ and some $1 < p < 2$,
\begin{equation}\label{conds}
S(u) \leq A\qquad{\rm and}\qquad  N_p(u) \leq B\ .
\end{equation}

\begin{remark}  Conditions (\ref{conds}) provide some ``uniform integrability control'' on the class of densities $u^4$ that satisfy them.
The proof that we give would yield similar results for essentially any other pair of conditions
that quantify uniform integrability.  The one we have chosen, moments and entropy, are natural in
PDE applications.  It is natural that some such condition is required: a bound on the deficit does not supply any compactness, as is clear from the scale-invariance.
\end{remark}

\noindent{\it Proof of Theorem~\ref{4thstab}:}
The proof is divided in several steps.

\noindent{\it $\bullet$ Step 1:} {\em  We show that $\|u\|_6$ cannot be too small provided $N_p(u)$ is not too large.}  Indeed,
$$\int_{B_R} |u|^4 \dd y \geq \|v\|_4^4 - R^{-p}N_p(u)\ .$$
Choosing $R>0$ such that $R^{-p}N_p(u) = \|v\|_4^4/2$ and using H\"older's inequality, we get
$$\frac{1}{2}\|v\|_4^4 \leq \int_{B_R} |u|^4 \leq \|u\|_6^2 (\pi R^2)^{2/3}\ ,$$
that is
\begin{equation}\label{fri0}
\|u\|_6^6 \geq c_1N_p(u)^{4/p}
\end{equation}
for some universal constant $c_1>0$.

\medskip
\noindent{\it $\bullet$ Step 2:} {\em To apply our previous results, we must multiply $u$ by a constant and rescale. In this step we show that these modifications do not
seriously affect the size of the deficit $\delta_{{\rm GNS}}[u]$.} 

Define
$$\widetilde u(y) :=   \frac{\|v\|_6}{\|u\|_6}\lambda^{1/3}u(\lambda y)\ .$$
where $\lambda$ is chosen so that $\sqrt{2}\|\nabla \widetilde u\|_2 = \|\widetilde u\|_4^2$. Note that $\|\widetilde u\|_6 = \|v\|_6$.
Since the rescaling does not affect the $L^6$ norm, it does not affect the deficit,
but the constant multiple does: we have
\begin{equation}\label{fo3}
\delta_{{\rm GNS}}[\widetilde u] =  \frac{\|v\|_6^3}{\|u\|_6^3}\delta_{{\rm GNS}}[u]\ ,
\end{equation}
By what we have noted in Step 1, we have an  {\it a-priori} upper bound on the factor
$\|v\|_6^3/\|u\|_6^3$ (see \eqref{fri0}), which gives the 
bound
\begin{equation}
\label{eq:control delta}
 \delta( \widetilde u) \leq C \delta_{{\rm GNS}}[u]\ .
\end{equation}

\medskip
\noindent{\it $\bullet$ Step 3:} {\em We now relate the constant multiple
and the scale factor when the deficit is small.}
First, we claim that
\begin{equation}\label{fri}
\left|\|  \widetilde u^4\|_4^4 - \|v\|_4^4\right| \leq C\delta_{{\rm GNS}}[u]\ .
\end{equation}
To see this note that, 
since $\sqrt{2}\|\nabla \widetilde  u\|_2 =  \|\widetilde  u\|_4^2$ (see Step 2), 
$$
\left|2\|  \widetilde u\|_4^6- \pi\| \widetilde u\|_6^6\right| = \delta( \widetilde u)\ ,
$$
The claim then follows by \eqref{eq:control delta} together with 
$$\pi\| \widetilde u\|_6^6 = \pi\| v\|_6^6 = 2\|v\|_4^6\ .$$
Let us also observe that, since
$$\|v\|_4^4 = \|u\|_4^4 =  \lambda^{2/3}\frac{\|u\|_6^4}{\|v\|_6^4}\|\widetilde u\|_4\ ,$$
by \eqref{fri} we also get
\begin{equation}\label{fri2}
|\lambda^{2/3}\|u\|_6^4 - \|v\|_6^4 | \leq C\delta_{{\rm GNS}}[u]\ .
\end{equation}

\medskip
\noindent{\it $\bullet$ Step 4:} {\em We now show that some translate $\hat u^4$ of $\widetilde u^4$  is close to $v^4$ when the deficit is small.}
Theorem \ref{thm:stability 6A} shows that there is a translate
$\hat u(y) = \widetilde u(y-y_0)$ of $\widetilde u$ such that
\begin{equation}\label{eq:hat u v 6}
\|\hat u^6 - v^6\|_1 \leq K_1 \delta_{{\rm GNS}}[\widetilde u]^{1/2}\ .
\end{equation}

Note that for positive numbers $a$ and $b$, 
$$|a^4- b^4| = |a-b|(a^3 + a^2b + a^2b + b^3) \quad{\rm and}\quad
|a^6- b^6| = |a-b|(a^5 + a^4b + a^3b^2 + a^2b^3 + ab^4 + b^5) \ .$$
Hence, since $(a^2+b^2)(a^3 + a^2b + a^2b + b^3) \leq 2(a^5 + a^4b + a^3b^2 + a^2b^3 + ab^4 + b^5)$,
it follows that
\begin{equation}\label{fri3}
|a^4- b^4|  \leq \frac{2}{a^2+b^2}|a^6 - b^6|\ . 
\end{equation}
So, observing that
$$
\frac{1}{u^2+v^2} \leq \frac{1}{v^2} \leq 1+R^2 \qquad \text{on }B_R,
$$
by \eqref{fri3}, \eqref{eq:hat u v 6}, and \eqref{eq:control delta}, we obtain
$$\int_{B_R}|\hat u^4 - v^4|\dd y \leq 2(1+R^2)\|\hat u^6 - v^6\|_1    \leq C (1+R^2)\delta_{{\rm GNS}}[u]^{1/2}\ .$$
Next, using (\ref{fri}),
\begin{eqnarray}
\int_{|x|\geq R} \hat u^4\dd y &=& \|\hat u\|_4^4 - \int_{|x|\leq R} \hat u^4\dd y\nonumber\\
&\leq & \|v\|_4^4 - \int_{|x|\leq R} \hat u^4\dd y + C \delta_{{\rm GNS}}[u]\nonumber\\
&=&   \int_{|x|> R} v^4\dd y +  \int_{|x|\leq R}(v^4-  \hat u^4)\dd y + C \delta_{{\rm GNS}}[u]\nonumber\\
&\leq &   \frac{\pi}{1+R^2} +  \int_{|x|\leq R}|v^4-  \hat u^4|\dd y + C \delta_{{\rm GNS}}[u]\nonumber\\
\end{eqnarray}
Combining results, we then get
$$ \|\hat u^4-v^4 \|_1 \leq   C(1+R^2)\delta_{{\rm GNS}}[u]^{1/2} + C(1+R^2)^{-1}\ ,$$
which (optimizing with respect to $R$) leads to the  estimate
\begin{equation}\label{fri5}
\|\widetilde u^4(\cdot - y_0)-v^4\|_1=\|\hat u^4-v^4 \|_1 \leq   C\delta_{{\rm GNS}}[u]^{1/4}\ .
\end{equation}

\medskip
\noindent{\it $\bullet$ Step 5:} 
Set $u_{1/\lambda}:=\lambda^{1/2}u(\lambda y)$.
Note that $\|u_{1/\lambda}\|_4=\|u\|_4$ and
$$
\int_{\R^2} |\widetilde u^4 - u_{1/\lambda}^4|\dd y
= \left|\frac{\|v\|_6^4}{\|u\|_6^4}\lambda^{-2/3}-1 \right|\|u\|_4^4\,.
$$
Now, by (\ref{fri2}), $\lambda^{2/3}\|u\|_6^4$ is uniformly bounded away from zero (for $\delta_{{\rm GNS}}[u]$
sufficiently small). Therefore
\begin{equation}
\label{eq:tilde u lambda}
\int_{\R^2} |\widetilde u^4 - u_{1/\lambda}^4|\dd y
\leq \frac{C \delta (u)}{\lambda^{2/3}\|u\|_6^4}\|u\|_4^4 \leq C \delta (u)\|u\|_4^4\ ,
\end{equation}
which combined with (\ref{fri5}) gives
\begin{equation}
\label{eq:bound 4 first}
\|u_{1/\lambda}^4(y-y_0)-v^4 \|_1 \leq   C\delta_{{\rm GNS}}[u]^{1/4}.
\end{equation}

\medskip
\noindent{\it $\bullet$ Step 6:} {\em We obtain upper and lower bounds on the scaling parameter $\lambda$.}  We already have an upper bound since 
(\ref{fri2}) says that $\lambda^{-1}\sim \|u\|_6^6$, and (\ref{fri0}) gives a lower bound for $\|u\|_6^6$ in terms of $N_p(u)$. 
Our assumption that $S(u)$, the entropy of $u^4$ (see \eqref{fo2}), is finite enters at this point. 

Since $\|v\|_4 = \pi/2$, for when $\delta_{{\rm GNS}}[u] < \delta_0$,  if follows from (\ref{eq:bound 4 first}) that
$$
\int_{B_1}\lambda^2u^4(\lambda(y-y_0))\dd y \geq \frac{\pi}{4},
$$
or equivalently
$$
\int_{B_{\lambda}(\lambda y_0)}u^4(y)\dd y \geq \frac{\pi}{4}.
$$
Thus, the average value of  $u^4$ on $B_{\lambda}(\lambda y_0)$ is at least $\lambda ^{-2}/2$. Hence by Jensen's inequality,
\begin{eqnarray*}
\frac{1}{\pi \lambda^2}\int_{B_{\lambda}(\lambda y_0)}u^4(y)\log(u^4(y))\dd y &\geq&
\left(\frac{1}{\pi \lambda^2}\int_{B_{\lambda}(\lambda y_0)}u^4(y)\dd y\right)\log\left(\frac{1}{\pi \lambda^2}\int_{B_{\lambda}(\lambda y_0)}u^4(y)\dd y\right)\\
&\geq& \frac{1}{4\lambda^2}\log\left( \frac{1}{4\lambda^2}\right),
\end{eqnarray*}
that is
$$
\int_{B_{\lambda}(\lambda y_0)}u^4(y)\log(u^4(y))\dd y \geq
\pi\left( -\frac{\log 2}{2} - \frac{\log \lambda}{2}\right)\ .
$$
Next we recall a standard estimate, valid for any non-negative integrable function $\rho$ on $\R^2$
with finite first moment (see for instance \cite[Lemma 2.4]{BCC}):
$$
\int_{\R^2}\rho(x)\,\log_- (\rho(x))\,\dd x \leq
\int_{\R^2} |x|\rho(x)\,\dd x + \frac 1e \int_{\R^2} e^{-|x|}\dd x=N_1(u)+\frac{2\pi}{e}\ ,
$$
where $\log_-(s)=\max\{-\log(s),0\}$.
Combining all the estimates together, we arrive at
\begin{equation}\label{sat6}
-\log \lambda \leq \frac{2}{\pi}(S(u) + N_1(u)) + \frac{4}{e} + \log 2\ .
\end{equation}
Since $N_1(u) \leq \|u\|_4^4 + N_p(u)$ for $p \geq 1$, the above inequality provides the desired lower bound on $\lambda$.

\medskip
\noindent{\it $\bullet$ Step 7: We now reabsorb $y_0$.}
Arguing as in the proof of Proposition \ref{prop:translation} and using \eqref{fri5},
\begin{eqnarray}
|y_0|\|v^4\|_4 &\leq&
 \left[N_p(\hat u)^{1/p}+N_p(v)^{1/p}\right]\|\hat u^4 -v^4\|_1^{1/q}\nonumber\\
 &\leq&
 C\left[N_p(\widetilde u)^{1/p}+|y_0|\|v^4\|_1^{1/p}+N_p(v)^{1/p}\right]\delta_{{\rm GNS}}[u]^{1/(4q)}\nonumber\\
 &\leq&
 C\left[\lambda^{-p-2/3}N_p( u)^{1/p}+|y_0|\|v^4\|_1^{1/p}+N_p(v)^{1/p}\right]
\delta_{{\rm GNS}}[u]^{1/(4q)}\ .\nonumber 
\end{eqnarray}
Using the bound (\ref{sat6}) this implies
 $|y_0| \leq C\delta_{{\rm GNS}}[u]^{1/(4q)}$. So, since
$$
\|\widetilde u^4 - \hat u^4\|_1=\|\widetilde u^4 - \widetilde u^4(\cdot - y_0)\|_1 \leq 4 |y_0|
\|\nabla \widetilde u\|_2\|\widetilde u\|_6^3
$$
and $\sqrt{2}\|\nabla \widetilde u\|_2=\|\widetilde u\|_4^2=\|v\|_4^2$,
as in the proof
of Proposition \ref{prop:translation} we get
$$
\|\widetilde u^4-v^4 \|_1 \leq   C\delta_{{\rm GNS}}[u]^{(p-1)/(4p)}\ .
$$
In particular, by \eqref{eq:tilde u lambda}  we obtain
$$
\|u_{1/\lambda}^4-v^4 \|_1 \leq   C\delta_{{\rm GNS}}[u]^{(p-1)/(4p)},
$$
which is equivalent to \eqref{infbndZ}.
\qed

\section{Application to stability for the Log-HLS inequality and to  Keller-Segel equation}
\label{sect:KS}

\subsection{{\it A-priori} estimates}\label{sect:apriori}

In this section we apply the results proved in the previous section, carrying out the strategy for quantitatively bounding the rate of approach to
equilibrium for critcal mass solutions of the Keller-Segel equation, and, along the way, proving a stability result for the Log-HLS inequality.
This and several other results obtained here may be of interest apart from their particular application to the Keller-Segel equation.

First of all, we recall some {\it a-priori}  regularity results concerning functions in level sets of the various functional $\F$, $\D$ and $\H_{\kappa,M}$
that have been defined in the introduction.

 As we have seen $\F[\sigma_{\kappa,M}] = \D[\sigma_{\kappa,M}] = 0$ for all $\kappa$ and $M$. But as $\kappa$ tends to $0$,
 $\sigma_{\kappa,M}(x)\dd x$ tend to a point mass (of mass $M$). Hence the level sets of neither $\F$ nor 
  $\D[\sigma_{\kappa,M}]$  are compact in $L^1(\R^2)$ or even uniformly integrable. It is also easy to see 
  that the level sets of $\H_{\kappa,M}$ are not compact in
  $L^1(\R^2)$, or even uniformly integrable. However, as shown in \cite{BCC}, taken together bounds on various combinations of 
  $\F[\rho]$, $\H_{\kappa,M}[\rho]$ and $\D[\rho]$ do yield strong estimates on $\rho$. 
  
 First, we recall that $\F$ and $\H_\kappa$ provide control of the entropy \cite[Theorem 1.9]{BCC}.
 Here and in the sequel, $\log_+$ denotes the positive part of the natural logarithm function.
  
  \begin{thm}[Entropy bound via $\F$ and $\H_\kappa$]\label{both}
Let $\rho$ be any density on $\R^2$ with mass $M=8\pi$, with $\H_{\kappa,8\pi}[\rho]
< \infty$ for some $\kappa > 0$.   Then there exist positive
computable constants $\gamma_1$ and $C_{\F\H}$, depending only
on $\kappa$ and $\H_{\kappa,8\pi}[\rho]$, such that
\begin{equation}\label{entbnd}
\gamma_1\int_{\R^2}\rho \log_+  \rho \dd x  \le \F[\rho]  +
C_{\F\H}\,.
\end{equation}
\end{thm}

Likewise, \cite[Theorem 1.10]{BCC} shows that a bound on $\F$, $\H_{\kappa,M}$ and $\D$ together controls the energy integral 
$\|\nabla u\|_2^2$.

\begin{thm}[Energy bound via $\F$, $\H_{\kappa,M}$ and ${\cal D}$]\label{both2}
Let $\rho$ be any density on $\R^2$  with mass $M=8\pi$ 
with $\F[\rho]$ finite, and $\H_{\kappa,8\pi}[\rho]$ finite for some $\kappa>0$.
Then there exist positive computable  constants $\gamma_2$
and $C_{\F\H\D}$, depending only on $\kappa$, $\H_{\kappa,8\pi}[\rho]$ and
$\F[\rho]$, such that
\begin{equation}\label{ccdbnd}
\gamma_2\,\int_{\R^2}|\nabla \rho^{1/4}|^2 \dd x \le \pi{\cal
D}[\rho] +C_{\F\H\D}\ \, .
\end{equation}
\end{thm}

Recall the classical Gagliardo-Nirenberg inequality
\begin{equation}\label{GNSalt}
  \int_{\R^2}|v|^{p}\dd x \le D_p\left[\int_{\R^2}|\nabla v|^{2}\dd x\right]^{p/2-2}\int_{\R^2}|v|^{4}\dd x\qquad \forall\, p\in[4,\infty)\, .
\end{equation}
Combining this with (\ref{ccdbnd}), we see that together $\F[\rho]$, $\H_\kappa[\rho]$ and $\D[\rho]$ give us a quantitative 
bound on $\|\rho\|_q$ for all $q< \infty$: 

\begin{cl}[$L^q$ bound $\F$, $\H_{\kappa,M}$ and ${\cal D}$]\label{lqb}
Let $\rho$ be any density on $\R^2$  with mass $M=8\pi$, with $\H_{\kappa,8\pi}[\rho]$ finite for some $\kappa>0$,
such that also  $\F[\rho]$ and  $\D[\rho]$ are finite. Then, for all $q\geq 1$, there is a constant $C$ depending only on $q$, $\kappa$, $M$,
$\F[\rho]$,  $\H_{\kappa,8\pi}[\rho]$ and  $\D[\rho]$, such that 
\begin{equation}\label{Lq}
\|\rho\|_q \leq C\ .
\end{equation}
\end{cl}

Next, with an argument analogous to the one used in \cite{BCC}, we can use the functional $\H_{\kappa,M}$ to control $p$th moments for all $p<2$:

\begin{thm}[Moments and lower bounds on the $L^{3/2}$-norm via $\H_{\kappa,M}$]\label{mothm}
Let  $\rho$ be a density on $\R^2$ with mass $M=8\pi$.
For all  $0\leq p < 2$, there is a constant $C$,
depending only on $p$ and $\kappa$, such that
\begin{equation}\label{moments}
\int_{\R^2}|x|^p\rho(x)\dd x \leq C \bigl(1+ \H_{\kappa,8\pi}[\rho] \bigr)\,,
\end{equation}
\begin{equation}\label{eq:norm32}
\|\rho\|_{3/2} \geq \frac{C}{\bigl(1+ \H_{\kappa,8\pi}[\rho] \bigr)^{1/p}}\,,
\end{equation}
\end{thm}
\begin{proof}
Since $\sigma_{\kappa,8\pi}$ has finite $p$th moments for all $p<2$, to prove \eqref{moments} it suffices to estimate
$$
\int_{\R^2}|x|^p\bigl|\rho(x)-\sigma_{\kappa,8\pi}(x)\bigr|\dd x.
$$
Observing that $|x|^p\leq C/\sqrt{\sigma_{\kappa,8\pi}}(x)$
and that
$$
\frac{\bigl|\rho-\sigma_{\kappa,8\pi}\bigr|}{\sqrt{\sigma_{\kappa,8\pi}}} \leq
\frac{\bigl|\sqrt\rho-\sqrt{\sigma_{\kappa,8\pi}}\bigr|^2}{\sqrt{\sigma_{\kappa,8\pi}}}
+2\sigma_{\kappa,8\pi}^{1/4} \frac{\bigl|\sqrt\rho-\sqrt{\sigma_{\kappa,8\pi}}\bigr|}{\sigma_{\kappa,8\pi}^{1/4}},
$$
we conclude easily using H\"older inequality.

Finally, \eqref{eq:norm32} is a consequence of \eqref{moments}: 
indeed, for any $\theta \in (0,1)$ and $q>1$,
\begin{eqnarray*}
\|\rho\|_1 &\leq& \left( \int_{\R^2} (1+|x|^p)\rho(x)\dd x \right)^{\theta}
\left( \int_{\R^2} \frac{\rho(x)}{(1+|x|^p)^{\theta/(1-\theta)}}\dd x \right)^{1-\theta}\\
&\leq& \left( \int_{\R^2} (1+|x|^p)\rho(x)\dd x \right)^{\theta}\|(1+|x|^p)^{-\theta/(1-\theta)}\|_{q/(q-1)}^{1-\theta} \|\rho\|_q^{1-\theta}
\end{eqnarray*}
Choosing $q=3/2$ and $\theta=1/(p+1)$ (so that $(1+|x|^p)^{-\theta/(1-\theta)}\in L^3(\R^2)$),
we get
$$
8\pi \leq C \|\rho\|_{3/2}^{p/(p+1)},
$$
which proves \eqref{eq:norm32}.
\end{proof}

 We close this subsection with the following observation that will be used later:
unlike $\F$ and $\D$, the functional $\H_{\kappa,M}$ is not scale invariant. Indeed, for any $M>0$, 
$\H_{\kappa,M}[\sigma_{\mu,M}] < \infty$ if and only if $\mu  =\kappa$.   In fact, later we shall need a somewhat more precise version of this estimate, which can be easily proved
by a direct computation: there exists a  constant $c_0>0$ depending only on $\kappa$ and $M$ such that
\begin{equation}\label{mango}
\int_{|y| < R}\frac{|\sqrt{\sigma_{\mu,M}}(y) - \sqrt{\sigma_{\kappa,M}}(y)|^2}{\sqrt{\sigma_{\kappa,M}}(y)}\dd y \geq c_0(\sqrt{\mu} - \sqrt{\kappa})^2\log(R)\,
\end{equation}

\subsection{A quantitative convergence result for the critical mass Keller-Segel equation}

We  now state and prove our a quantitative bound on the rate of relaxation to equilibrium for the critical mass Keller-Segel equation. 

\begin{thm} 
\label{thm:KS}
Let $\rho(x,t)$ be any properly dissipative solution of the Keller-Segel equation of critical mass $M = 8\pi$ in the sense of
 \cite{BCC}, so that in particular
$\H_{\kappa,8\pi}[\rho(\cdot,0)] < \infty$ for some $\kappa>0$,  and  $\F[\rho(\cdot,0)]<\infty$. Let us suppose also that $\int_{\R^2}x\rho(x,0)\dd x =0$.
Then, for all $\epsilon>0$,  there are constants $C_1$ and $C_2$, depending only on $\epsilon$, $\kappa$,  $\H_{\kappa,8\pi}[\rho(\cdot,0)]$ and  $\F[\rho(\cdot,0)]$,
such that, for all $t>0$, 
\begin{equation}\label{decay1}
\F[\rho(\cdot,t)] - C(8\pi)  \leq C_1 (1+t)^{-(1-\epsilon)/8} 
\end{equation}
\begin{equation}\label{decay2}
\inf_{\mu > 0}\| \rho(\cdot,t) - \sigma_{\mu,8\pi}\|_1  \leq C_2 (1+t)^{-(1-\epsilon)/160}.
\end{equation}
Moreover, there is a positive number
$a>0$, depending only on  $\H_{\kappa,8\pi}[\rho(\cdot,0)]$ and  $\F[\rho(\cdot,0)]$, so that
for each $t > 0$, 
$$\inf_{\mu > 0}\| \rho(\cdot,t) - \sigma_{\mu,8\pi}\|_1  = \min_{a < \mu <1/a }\| \rho(\cdot,t) - \sigma_{\mu,8\pi}\|_1.$$
Finally, for each $t>0$, the above minimum is achieved at value $\mu(t)$ such that 
\begin{equation}\label{decay3}
(\mu(t) - \kappa)^2  \leq \frac{C}{\log(e+t)}\ .
\end{equation}
In particular
\begin{equation}\label{decay4}
\| \rho(\cdot,t) - \sigma_{\kappa,8\pi}\|_1 \leq  \frac{C}{\sqrt{\log(e+t)}}\,.
\end{equation}
\end{thm}

As indicated in the introduction, to carry out the proof of Theorem!~\ref{thm:KS}, we need an ``almost Lipschitz'' property of the
functional $\F$. We introduce this next, before turning to the proof of  Theorem~\ref{thm:KS}.

To obtain continuity properties of the Log-HLS functional, we will need to impose some restrictions
on the set of densities.  
In view of the wider interest of the almost Lipschitz continuity of the entropic part of $\F$
(Theorem~\ref{entcont}  below),
our next definition refers to densities on $\R^n$.

\begin{defi}\label{goset} For $p > 0$, $q>1$ and $A,B < \infty$,  let ${\mathcal M}_{n,p,q,A,B}$ denote the set of mass densities $\rho$ on 
$\R^n$ such that
\begin{equation}\label{bnds1}
 \int_{\R^n} |x|^p \rho(x) \dd x  \leq  A  \qquad {\rm and} \qquad  \int_{\R^n}| \rho(x)|^q \dd x  \leq  B\ .
 \end{equation}
\end{defi}

Note that we {\em do not} specify the mass of the densities in ${\mathcal M}_{p,q,A,B}$ though of course they can be bounded above in terms of 
$A$ and $B$. 
A key result (which will be proved later in Subsection \ref{sect:cont}) is that, for any $p$, $q$, $A$ and $B$, the Log-HLS functional is almost Lipschitz continuous on 
${\mathcal M}_{2,p,q,A,B}$:

\begin{thm}\label{loghlsfunccont}  For all $0< \epsilon< 1$, and all $M>0$, 
Then  there is a constant $C$ depending only on $\epsilon$, $M$,  $p$, $q$, $A$ and  $B$  such that for any
$\rho,\sigma \in {\mathcal M}_{2,p,q,A,B}$ both of mass $M$, 
$$\left| \F[\rho] - \F[\sigma]\right| \leq  C \|\rho - \sigma\|_1 ^{1-\epsilon}\ .$$
\end{thm}

\noindent{\it Proof of Theorem~\ref{thm:KS}:} 
Of course it suffices to prove all of the  estimates  in Theorem~\ref{thm:KS} for $t$ large.

As shown in \cite{BCC}, for all $t_0>0$, and all
$p< 2$ and $q<\infty$, there exist finite constants $A$ and $B$ such that 
$${\rm for\ all}\quad t \geq t_0\ , \qquad \rho(\cdot,t) \in   {\mathcal M}_{2,p,q,A,B}\,, $$
see also Subsection \ref{sect:apriori}.

Choose $t_0=1$.
As noted earlier in this section (see \eqref{decay}),
by the definition of properly dissipative solution
\begin{equation}\label{entdec2}
\H_{\kappa,8\pi}[\rho(\cdot,T)] + \int_0^T \D[\rho(\cdot,t)] \dd t  \leq \H_{\kappa,\pi}[\rho(\cdot,0)]\ , 
\end{equation}
we immediately deduce that, for all $T \geq 1$,
\begin{equation}\label{entdec2Z}
\frac{1}{T - 1}  \int_{1}^T \D[\rho(\cdot,t)] \dd t  \leq \frac{1}{T - 1} \H_{\kappa,\pi}[\rho(\cdot,0)]\ .
\end{equation}

Then, Theorem \ref{mothm} together with \eqref{entdec2}
ensures a uniform bound on $\int_{\R^2}|x|^p \rho(x,t)\dd x$ for any $p<2$,
which also ensure a lower bound on $\|\rho(\cdot,t)\|_{3/2}$.

Hence, by  (\ref{pass}) and \eqref{entdec2Z},
we deduce the existence of some $1\leq t \leq T$,
$$
\delta[\rho^{1/4}(\cdot,t)] \leq   \frac{C}{T}   \H_{\kappa,\pi}[\rho(\cdot,0)]\ .
$$

Next, Theorem~\ref{both} gives us an {\it a-priori} 
upper bound on the entropy $S[\rho(\cdot,t)]$, and thus permits us to apply 
Theorem \ref{4thstab} for $T$ sufficiently large.
We conclude that, for any $p<2$, there exist $a>0$,
 \textit{some} $\mu\in[a,1/a]$, and {\em some} $1 \leq t \leq T$, such that
\begin{equation}\label{upbnd1Z}
\|\rho(\cdot,t) - \sigma_{\mu,8\pi}\|_1  \leq C \left(\frac{1}{T}   \H_{\kappa,\pi}[\rho(\cdot,0)] \right)^{(p-1)/4p}\ ,
\end{equation}
(recall that the density $v_\lambda^4$ is a multiple of some $\sigma_{\mu,8\pi}$).

Next, since we can choose $p$ arbitrarily close to $2$,
by Theorem~\ref{loghlsfunccont}
$$
\F[\rho(\cdot,t)] - C(8\pi)  \leq C T^{-(1-\epsilon)/8}
$$
for {\em some} $1 \leq t \leq T$. However we can now use that 
$\F[\rho(\cdot,t)]$ is monotone decreasing to deduce that
\begin{equation}\label{upbnd1ZW}
\F[\rho(\cdot,T)] - C(8\pi)  \leq C T^{-(1-\epsilon)/8} \ 
\end{equation}
\emph{for all} $T$ sufficiently large. 
Hence, up to
adjusting the constant $C$ we obtain (\ref{decay1}).

We may now apply Theorem~\ref{loghlsstab} to conclude that for 
 {\em all}   $t>0$, there is some $\mu=\mu(t)\in[a,1/a]$ such that
\begin{equation}\label{upbnd1ZZ}
\|\rho(\cdot,t) - \sigma_{\mu,8\pi}\|_1  \leq C (1+t)^{-(1-\epsilon)/160} \ .
\end{equation}

Finally, we use the bound on $\H_{\kappa,8\pi}[\rho(\cdot,t)]$ to fix the scale:
Since $|\a^2-\b^2|^2 \leq |\a^4-\b^4|$, \eqref{upbnd1ZZ} implies
\begin{equation}\label{mango1}
\|\sqrt{\rho(\cdot,t)} - \sqrt{\sigma_{\mu,8\pi}}\|_2^2 \leq   C (1+t)^{-(1-\epsilon)/160} 
\end{equation}
By the triangle inequality, (\ref{mango}), (\ref{mango1}), and using that $ \sqrt{\sigma_{\mu,8\pi}} \geq \kappa^{1/2}(\kappa + R^2)^{-1}$ inside $B_R$, we get
\begin{eqnarray}
\sqrt{\H_{\kappa,8\pi}[\rho(\cdot,t)]} 
&\geq&  
\left(\int_{|y| < R}\frac{\sqrt{\rho(y,t)} - \sqrt{\sigma_{\kappa,8\pi}(y)}}{\sqrt{\sigma_{\kappa,8\pi}(y)}}\dd y\right)^{1/2}\nonumber\\
&\geq&  
\left(\int_{|y| < R}\frac{\sqrt{\sigma_{\mu,8\pi}(y)} - \sqrt{\sigma_{\kappa,8\pi}(y)}}{\sqrt{\sigma_{\kappa,8\pi}(y)}}\dd y\right)^{1/2}- 
C\sqrt{\frac{\kappa+R^2}{\kappa^{1/2}}} (1+t)^{-(1-\epsilon)/320}\nonumber\\
&\geq&  c(\kappa-\mu)^2\log(R)  -C\sqrt{\frac{\kappa+R^2}{\kappa^{1/2}}} (1+t)^{-(1-\epsilon)/320} \ ,\nonumber
\end{eqnarray}
where in the last line we have used (\ref{mango}) and the fact that we have an {\it a-priori} lower bound on $\mu$.

Thus,
$$
(\kappa-\mu)^2 \leq \frac{C}{\log(R)} \left[\sqrt{\H_{\kappa,8\pi}[\rho(\cdot,0)]}  +\sqrt{\frac{\kappa+R^2}{\kappa^{1/2}}} (1+t)^{-(1-\epsilon)/320}  \right]
$$
Choosing $R= (e+t)^{(1-\epsilon)/320}$ we get
$$
(\kappa-\mu)^2  \leq \frac{C}{\log(e+t) },
$$
as desired.
Finally \eqref{decay4} follows from \eqref{decay2} and \eqref{decay3},
observing that
$$
\|  \sigma_{\mu,8\pi}-\sigma_{\kappa,8\pi}\|_1 \leq C_\kappa |\mu -\kappa| \qquad \forall\, \mu >0\,,
$$
with $C_\kappa$ depending on $\kappa$ only.
\qed

\subsection{Stability for the Logarithmic HLS inequality: proof of Theorem~\ref{loghlsstab}}

We now prove a stability result for the  Log-HLS
inequality. 
The proof of Theorem~\ref{loghlsstab} is based on the recently discovered fact \cite{CCL} that $\F$ is decreasing along the fast diffusion flow. 
Moreover, since the fast diffusion flow is gradient flow for $\H_{\kappa,M}$,
also $\H_{\kappa,M}$ is decreasing along the fast diffusion flow. While 
$\D$ is not decreasing along the flow, the dissipation relation gives us 
\begin{equation}\label{disrel}
\int_0^T \D[\sigma(\cdot ,t)] \dd t \leq \F[\rho]- C(M)\ ,
\end{equation}
where $\sigma(x,t)$ is the solution to (\ref{fdA}) with initial data $ \sigma(\cdot,0)=\rho$.  The  estimate (\ref{disrel}) is proved in \cite{CCL} for initial data $\rho$
such that, for some $C, R>0$, $\rho(x) \leq C|x|^{-4}$ for all $|x|>R$. Then, regularity estimates from \cite{BV} permit one to integrate by parts and prove that $\lim_{t\to\infty}\F[\rho(\cdot,t)] = C(M)$, which leads to (\ref{disrel}). However, the regularity provided by \cite{BV} is only used in 
a qualitative way, and the values of $R$ and $C$ do not matter.  Hence, a simple truncation and replacement argument can be used to achieve these bounds while making an arbitrarily small effect on $\F[\rho]$ and $\H_\kappa[\rho]$, and moving $\rho$ an arbitrarily small distance in the $L^1$ norm. So, we may freely assume the bound  $\rho(x) \leq C|x|^{-4}$ for all $|x|>R$ for some finite 
constants $C$ and $R$.

\medskip

\noindent{\it Proof of Theorem~\ref{loghlsstab}:}  Let $\sigma(\cdot,t)$ be the solution of (\ref{fdA}) with initial data $\rho$. As
as explained above the dissipation relation \eqref{disrel} holds, therefore
\begin{equation}\label{disslem2}
\delta_{\rm HLS}[\rho]  \geq \int_0^T \D[ \sigma(\cdot,t)]\dd t\qquad \forall\,T>0\ .
\end{equation}

We proceed form here  in several steps: 

\medskip

\noindent{\it $\bullet$ Step 1: The HLS deficit of $\rho$ controls the GNS deficit of $\sigma(\cdot,t)$ for some $t$ close to $0$.} 
Pick some $T \in (0,1]$, with $\delta_{\rm HLS}[\rho] \ll T \ll 1$, to be chosen later. 
 Then by  \eqref{disslem2},
there exists some $t \in (0,T)$ such that
\begin{equation}\label{goodset} 
 \D[\sigma(\cdot,t)] \leq \frac{ \delta_{\rm HLS}[\rho]}{T}\ 
 \end{equation}

Since $\H_{\kappa,M}[\sigma(\cdot,t)]$  is decreasing along the flow, Theorem~\ref{mothm} gives us a lower bound on $\|\sigma(\cdot,t)\|_{3/2}$. Then, by (\ref{pass}), there is a constant $C>0$
such that 
$$
{\mathcal D}[\sigma(\cdot,t)]  \geq C\delta[\sigma^{1/4}(\cdot,t)]\ .
$$
Hence, by \eqref{goodset}  we get
 $$
 \delta[\sigma^{1/4}(\cdot,t)] \leq C   \frac{  \delta_{\rm HLS}[\rho]}{T}\  .
 $$

\medskip

\noindent{\it $\bullet$ Step 2: Application of stability for the GNS  inequality.}  Recalling 
that $v_\lambda^4$ is a multiple of $\sigma_{\mu,M}$ for some $\mu$,  by
Theorem \ref{4thstab} and Step 1, there exists some
$\mu>0$ (on which we have {\it a-priori} bounds above and below) such that
$$\| \sigma(\cdot,t)-\sigma_{\mu,M}  \|_1  \leq C\left( \frac{\delta_{\rm HLS}[\rho]}{T}  \right)^{(p-1)/4p} \ .$$
Hence, by the triangle inequality,
\begin{equation}\label{app55}
\| \rho-\sigma_{\mu,M} \|_1  \leq  C \|  \rho-\sigma(\cdot,t) \|_1 +   C\left( \frac{  \delta_{\rm HLS}[\rho]}{T}\right)^{(p-1)/4p} \ .
\end{equation}

\noindent{\it $\bullet$ Step 3: Controlling $\|\rho - \sigma(\cdot,t)\|_1$.} We claim that for all $\epsilon>0$ and all $1< p < 2$, 
there is a constant $C$ such that 
\begin{equation}\label{app56}
 \|  \rho-\sigma(\cdot,t) \|_1  \leq  C\left(  1 + \left( \frac {\delta_{\rm HLS}[\rho]}{T}\right)^{p/4(p+1)}\right) T^{p(1-\epsilon)/4(p+1)} + 
CT^{p(1-\epsilon)/8(p+1)}\ .
 \end{equation}
This will be proven below.
Assuming this for now, we complete the proof in the next step.

\noindent{\it $\bullet$ Step 4: Optimizing in $T$.} 
Combining \eqref{app55} and \eqref{app56}, we get
$$
\| \rho- \sigma_{\mu,M} \|_1 \leq C\left( \frac{  \delta_{\rm HLS}[\rho]}{T}\right)^{(p-1)/4p}+ 
CT^{p(1-\epsilon)/8(p+1)} .
$$
Setting ${\displaystyle r = \frac{p(1-\epsilon)}{8(p+1)}}$ and  ${\displaystyle s = \frac{p-1}{4p}}$,
we choose $T:= \delta_{\rm HLS}[\rho]^{s/(r+s)}$ to obtain
$$
\| \rho-\sigma_{\mu,M} \|_1 \leq C\delta_{\rm HLS}[\rho]^{rs/(r+s)}.
$$
Since $p$ can be chosen arbitrarily close to $2$, we obtain the result. \qed

\medskip
We close this subsection by proving (\ref{app56}). To this aim,
we make use of the fact that, for each $\kappa>0$,  the equation (\ref{fdA}) is a gradient 
flow of
the functional 
 $\H_{\kappa,M}$, with respect to the $2$-Wasserstein metric ${\rm W}_2$, on the space
of densities of mass $M$. This has the standard consequence that 
 \begin{equation}
\label{app45}
{\rm W}_2^2(\sigma(\cdot,s), \sigma(\cdot,t)) \leq \H_{\kappa,M}[\rho] (t-s)
\end{equation}
for all $t>s \geq 0$, see for instance \cite{AGS} or also \cite[Lemma 5.3]{BCC}.
That is, the fact that the equation is gradient flow for the $2$-Wasserstein metric automatically yields a H\"older$-1/2$ modulus of continuity bound in this metric. What we need now
is to improve this bound into a $L^1$ continuity. 

In the proof, we use (\ref{app45}) together with the following interpolation 
result, see \cite[Theorem 5.11]{BCC}:

\begin{thm}[Interpolation bound]\label{inerpbnd}
Let $\sigma_0$ and $\sigma_1$ be two densities of mass $M$ on
$\R^2$ such that for some $q>2$, $\|\sigma_0\|_{q+1}^{q+1}\ ,\
\|\sigma_1\|_{q+1}^{q+1} \leq K$. Suppose also that
$\sigma_0^{1/4}$ and  $\sigma_1^{1/4}$ have square integrable
distributional gradients.  Then
\begin{eqnarray}
\|\sigma_0 - \sigma_1\|_2^2 &\leq& \left(  \|\nabla
(\sigma_0^{1/4})\|_2 + \|\nabla (\sigma_1^{1/4})\|_2\right)
(2^{5/2} + 2^{9/2}K)
 (\W_2(\sigma_0, \sigma_1))^{(4q-3)/(4q+2)}\nonumber\\
 &+& 16 M^{(q-1)/q}K^{(q+2)/2q} (\W_2(\sigma_0, \sigma_1))^{(q-1)/(2q+1)}\ .\nonumber
 \end{eqnarray}
\end{thm}

\noindent{\it Proof of \eqref{app56}:} We apply the interpolation bound quoted above with $\sigma_0 = \rho$ and $\sigma_1 = \sigma(\cdot,t)$. 
By Theorem \ref{both2} and (\ref{goodset}), we have 
$$\|\nabla \sigma_0^{1/4}\|_2^2 \leq C\,, \qquad \|\nabla \sigma_1^{1/4}\|_2^2 \leq C\left(1 + \frac {\delta_{\rm HLS}[\rho]}{T}\right)\ ,$$
where $C$ depends only on $\kappa$, $B_\F$, $B_\H$ and $B_\D$. As noted above, thanks to (\ref{GNSalt}), this  gives 
$\|\sigma_0\|_{q+1}^{q+1}\leq K$, for a constant $K$ depending only on $q$, $\kappa$, $B_\F$, $B_\H$ and $B_\D$.
Then, using the fact that the evolution (\ref{fdA}) is a bounded in all $L^p$ norms and is uniformly bounded for $t\in [0,1]$, we obtain such a bound also for 
$\|\sigma_1\|_{q+1}^{q+1}$ (recall that, by assumption, $t \in [0,T]\subset [0,1]$).

By (\ref{app45}), $\W_2(\sigma_0, \sigma_1) \leq C\sqrt{T}$.  Pick $\epsilon>0$, and choose $q$ so large that 
$$\frac{4q-3}{4q+2} > 1-\epsilon\qquad{\rm and}\qquad \frac{q-1}{2q+1} >  \frac{1-\epsilon}{2}\ .$$
We then have
$$\|\sigma_0 - \sigma_1\|_2 \leq C\left(  1 + \left( \frac {\delta_{\rm HLS}[\rho]}{T}\right)^{1/4}\right) T^{(1-\epsilon)/4} + 
CT^{(1-\epsilon)/8}\ .$$

Next, we estimate, for non-negative functions $f$ on $\R^2$, and all $R>0$,
\begin{eqnarray}
\|f\|_1  &=& \int_{|x| \le R} f(x)\dd x + \int_{|x| \ge R} f(x)\dd x\nonumber\\
&\le& (\pi R^2)^{1/2}\|f\|_2 + \frac{1}{R^p}\int_{\R^2}|x|^pf(x)\dd x\ .\nonumber
\end{eqnarray}
Optimizing in $R$ yields
$$\|f\|_1 \leq C\|f\|_2^{p/(p+1)} \left(\int_{\R^2}|x|^p |f(x)|\dd x\right)^{1/(1+p)}\  ,$$
Applying this, we finally obtain
$$
\|\sigma_0 - \sigma_1\|_1 \leq C\left(  1 + \left( \frac {\delta_{\rm HLS}[\rho]}{T}\right)^{p/4(p+1)}\right) T^{p(1-\epsilon)/4(p+1)} + 
CT^{p(1-\epsilon)/8(p+1)}\ .$$

\qed

\subsection{Continuity properties of the Log-HLS functional:
proof of Theorem \ref{loghlsfunccont}}\label{sect:cont}

In order to prove Theorem \ref{loghlsfunccont},
we begin with a continuity result for the entropy that is of interest in its own right.    
In this section, $p$, $q$, $A$ and $B$ are as in Definition~\ref{goset}.

The next result states the almost Lipschitz continuity of the entropic part of $\F$,
and is not restricted to dimension two.
 
\begin{thm}\label{entcont}  
Then  there is a constant $C$, depending only on $n$, $p$, $q$, $A$ and  $B$,
such that for any
$\rho,\sigma \in {\mathcal M}_{n,p,q,A,B}$,
$$\left|\int_{\R^n} \rho \log \rho(x)\dd x -   \int_{\R^n} \sigma \log \sigma(x)\dd x\right| \leq  C \|\rho - \sigma\|_1 \log\| \rho - \sigma\|_1\ .$$
\end{thm}

We first prove two lemmas.

\begin{lm} For all $0<s<p$ there is a  constant $C$, depending only on $s$, $n$, $p$, $q$, $A$ and $B$, such that  
 $$\int_{|x|> R} \rho |\log \rho|(x)\dd x  \leq CR^{-s}\ ,$$
for all $\rho \in   {\mathcal M}_{n,p,q,A,B}$ and all $R>0$.
 \end{lm}

\begin{proof}
 We begin by recalling the following elementary inequality:
 for all $r>0$,
 \begin{equation}
 \label{eq:log r}
 |\log s| \leq \frac1r \max \{s^r,s^{-r}\} \qquad \forall\, s>0\,.
 \end{equation}

Now, pick $0< \gamma < 1$ and set
$$r := (q-1)\gamma>0\ .$$
We claim that
\begin{equation}\label{bnds2}
 \int_{\R^n} |x|^{p(1-\gamma)} \rho^{1+r}(x)\dd x \leq  \left( \int_{\R^n} |x|^p \rho(x) \dd x\right)^{1-\gamma} \left(\int_{\R^n}| \rho(x)|^q \dd x\right)^{\gamma}\ .
 \end{equation}
 To see this, define $\beta = q\gamma$ and note that $\beta+(1-\gamma) = 1+r$. Thus by H\"older's inequality,
 \begin{eqnarray}
 \int_{\R^n} |x|^{p(1-\gamma)} \rho^{1+r}(x)\dd x &=& \int_{\R^n} |x|^{p(1-\gamma)} \rho^{1-\gamma}(x)\rho^\beta(x)\dd x\nonumber\\
 &\leq&  \left(\int_{\R^n} (|x|^{p(1-\gamma)} \rho^{1-\gamma}(x))^{1/(1-\gamma)}\dd x\right)^{1-\gamma}
  \left( \int_{\R^n}  (\rho^\beta(x))^{1/\gamma}\dd x\right)^{\gamma}\ ,
  \nonumber
 \end{eqnarray}
 from which the claim follows. 
 
Thus, under the conditions (\ref{bnds1}), 
with $\gamma$ and $r$ chosen as above we have a uniform bound on $ \int_{\R^n} |x|^{p(1-\gamma)} \rho^{1+r}(x)\dd x$. Hence, for $r$ and $\gamma$ chosen as above, using
\eqref{eq:log r} with $s=\rho(x)$
we get
 \begin{eqnarray}
 \label{eq:bound 1}
  \int_{\{|x| \geq R\} \cap\{\rho\geq 1\}}  \rho|\log \rho| (x) \dd x &\leq& \frac1r \int_{|x|> R} \rho^{1+r}(x)\dd x\nonumber\\
  &=& \left(\frac{1}{r}\int_{|x|> R} |x|^{p(1-\gamma)}\rho^{1+r}(x)\dd x\right)R^{-p(1-\gamma)}.
  \end{eqnarray}

Next, we want to consider the set $\{\rho \leq 1\}$. Pick $0< \delta < 1$ and set
$$\alpha := p(1-\delta)\ .$$
We shall require $\alpha > n\delta$ (or equivalently $p/n > \delta/(1-\delta)$), which is always satisfied for $\delta$ sufficiently small.

Then, by \eqref{eq:log r} (with $r=\delta$ and $s=\rho(x)$) and H\"older's inequality, for all $\alpha>n\delta$ we have
\begin{eqnarray}
  \int_{\{|x| \geq R\} \cap\{\rho\leq 1\}}  \rho|\log \rho| (x) \dd x &\leq& \frac1\delta \int_{|x|> R} \rho^{1-\delta}(x)\dd x\nonumber\\
  &=& \frac{1}{\delta}\int_{|x|> R} |x|^{\alpha}\rho^{1-\delta}(x)|x|^{-\alpha}\dd x\nonumber\\
   &\leq& \frac{1}{\delta}\left(
   \int_{\R^n}  |x|^{\alpha/(1-\delta)}\rho(x) \dd x\right)^{1-\delta}  \left( \int_{|x|> R}  |x|^{-\alpha/\delta}\dd x\right)^\delta \nonumber\\
   &=& \frac{1}{\delta}\left(
   \int_{\R^n}  |x|^{p}\rho(x) \dd x\right)^{1-\delta}  \left( \int_{|x|> R}  |x|^{-\alpha/\delta}\dd x\right)^\delta. \nonumber
  \end{eqnarray}
  Computing
  $$\int_{|x|> R}  |x|^{-\alpha/\delta}\dd x =
  \frac{n|B_1^n|\delta}{\alpha -n\delta}R^{-(\alpha - n\delta)/\delta}$$
  and recalling the definition of $\alpha$, we obtain
$$\int_{\{|x| \geq R\} \cap\{\rho\leq 1\}}  \rho|\log \rho| (x) \dd x  \leq 
\frac{1}{\delta}\left(
   \int_{\R^n}  |x|^{p}\rho(x) \dd x\right)^{1-\delta} \left( \frac{n|B_1^n|r}{\alpha -n\delta}\right)^\delta R^{-(p(1-\delta) -n \delta)}\ .$$
   (Here and in the sequel, $|B_1^n|$ denotes the Lebesgue measure of the unit ball in $\R^n$.)
 Combining this bound with \eqref{eq:bound 1} and choosing both $\gamma$ and $\delta$ sufficiently small,  we have the result. 
 \end{proof}

\begin{lm}  For all $0<t<q$, $\rho \in   {\mathcal M}_{n,p,q,A,B}$, and $R>0$, and for all $0 < \epsilon < 1/e$,
it holds
\begin{multline}\label{cdg2}
\left|\int_{|x| \leq R} \rho \log \rho(x)\dd x -   \int_{|x|\leq R} \sigma \log \sigma(x)\dd x\right| \leq \\
2|\log \epsilon| \left[  2\epsilon |B_1^n|R^n + \frac{2}{t-1} \epsilon^{q-t} (|\rho\|_q^q +\|\sigma\|_q^q) + \|\rho- \sigma\|_1\right] \ .
\end{multline}
\end{lm}

\begin{proof} Pick $\epsilon$ with $0 < \epsilon < 1/e$ and define
$$\rho_\epsilon(x) :=
\begin{cases}  
\epsilon & \text{if }  \rho(x) \leq \epsilon\\
 \rho(x) &  \text{if } \epsilon < \rho(x) \leq 1/\epsilon\\
1/\epsilon &  \text{if }  \rho(x) \geq 1/\epsilon\end{cases}
\ .$$

Note that $s\mapsto s \log s$ is decreasing on $(0,1/e)$, so
$$
\bigl|s\log s - \epsilon \log \epsilon \bigr| \leq \epsilon \bigl| \log \epsilon \bigr| \qquad \forall \, s \in (0,\epsilon]\, .
$$
Therefore,  
$$\left|\int_{ \{|x| \leq R\}\cap\{\rho \leq \epsilon\}} \rho \log \rho(x)\dd x -  \int_{ \{|x| \leq R\}\cap\{\rho \leq \epsilon\}}\rho_\epsilon \log \rho_\epsilon(x)\dd x\right| \leq \epsilon |\log \epsilon| |B_1^n|R^n \ .$$
In an analogous way,
$$\left|\int_{ \{|x| \leq R\}\cap\{\rho \leq \epsilon\}}  \rho(x)\dd x -  \int_{ \{|x| \leq R\}\cap\{\rho \leq \epsilon\}}\rho_\epsilon(x)\dd x\right| 
\leq \epsilon |B_1^n|R^n \ .$$
Next, by Chebychev's inequality, $|\{ \rho\geq 1/\epsilon\}| \leq \epsilon^q\|\rho\|_q^q$.
Thus, for any $1 < t < q$, applying \eqref{eq:log r} with $r=t-1$ we get
\begin{eqnarray}
\int_{\{|x| \leq R\}\cap \{ \rho\geq 1/\epsilon\}} \rho \log \rho(x)\dd x 
&\leq& 
\frac{1}{t-1}  \int_{\{|x| \leq R\}\cap \{ \rho\geq 1/\epsilon\}}  \rho^t(x)\dd x\nonumber\\
&\leq& 
\frac{1}{t-1} \|\rho\|_q^t \left( |\{ \rho\geq 1/\epsilon\}\|\right)^{(q-t)/q}  \nonumber\\
&\leq& 
\frac{1}{t-1} \epsilon^{q-t}\|\rho\|_q^q \ . \nonumber
 \end{eqnarray}
Hence, 
 $$\left|\int_{ \{|x| \leq R\}\cap\{\rho \geq 1/\epsilon\}} \rho \log \rho(x)\dd x -  \int_{ \{|x| \leq R\}\cap\{\rho \geq 1/\epsilon\}}\rho_\epsilon \log \rho_\epsilon(x)\dd x\right| \leq \frac{1}{t-1} \epsilon^{q-t}\|\rho\|_q^q \ .$$
In a similar way,
 $$\left|\int_{ \{|x| \leq R\}\cap\{\rho \geq 1/\epsilon\}} \rho (x)\dd x -  \int_{ \{|x| \leq R\}\cap\{\rho \geq 1/\epsilon\}}\rho_\epsilon (x)\dd x\right| \leq  \epsilon^{q-1}\|\rho\|_q^q \ .$$
 
Thus, combining all these estimates together, we have
\begin{equation}\label{cdg0}
\left|\int_{ \{|x| \leq R\}} \rho \log \rho(x)\dd x -  \int_{ \{|x| \leq R\}}\rho_\epsilon \log \rho_\epsilon(x)\dd x\right| \leq 
\epsilon |\log \epsilon| |B_1^n|R^n + \frac{1}{t-1} \epsilon^{q-t}\|\rho\|_q^q 
\ ,
\end{equation}
and
\begin{equation}\label{cdg1}
\int_{|x|\leq R}|\rho - \rho_\epsilon|\dd x  \leq    \epsilon |B_1^n|R^n + \epsilon^{q-1}\|\rho\|_q^q\ .
\end{equation}
Of course, we have the analogous estimates for $\sigma$. 

Next, we observe that the derivative of $s\mapsto s \log s$ on $[\epsilon,1/\epsilon]$
is bounded by $2 |\log \epsilon|$.
Hence,
since $\rho_\epsilon$ and $\sigma_\epsilon$ are bounded below by $\epsilon$ and above by $1/\epsilon$, 
$$|\rho_\epsilon \log \rho_\epsilon(x) - \sigma_\epsilon \log \sigma_\epsilon(x)| \leq 2|\log \epsilon| |\rho_\epsilon(x) - 
\sigma_\epsilon(x)| \qquad \forall\, x\in \R^n\ .$$
Integrating over $\{|x| \leq R\}$ we find
$$\int_{\{|x| \leq R\}} |\rho_\epsilon \log \rho_\epsilon(x) - \sigma_\epsilon \log \sigma_\epsilon(x)| \leq 2|\log \epsilon|
\int_{\{|x| \leq R\}} |\rho_\epsilon(x) - 
\sigma_\epsilon(x)| \dd x\ .$$ 
Combining this with (\ref{cdg0}),  (\ref{cdg1}), and the corresponding estimates for $\sigma$, we obtain \eqref{cdg2}.
\end{proof}

\medskip

\noindent{\it Proof of Theorem~\ref{entcont}:} Combining the last two lemmas,
for any $s \in (0,p)$, $t \in (1,q)$ and $R>0$, we have
$$\int_{\R^n} |\rho \log \rho(x) - \sigma \log \sigma(x)| \leq CR^{-s} + 
2|\log \epsilon| \left[ 2 \epsilon |B_1^n|R^n + \frac{2}{t-1} \epsilon^{q-t} (|\rho\|_q^q +\|\sigma\|_q^q) + \|\rho- \sigma\|_1\right] 
\ .$$ 
Choosing $R=\left(\epsilon|\log \epsilon|\right)^{-1/(n+s)}$,
and recalling that we can take $s$ close to $p$ and $t$ close to $1$, we obtain
$$\int_{\R^n} |\rho \log \rho(x) - \sigma \log \sigma(x)| \leq C \epsilon^m + |\log \epsilon| \|\rho - \sigma\|_1\ ,$$
for any $0<m<\min\left\{ \frac{p}{n+p}\ , \ q-1\ \right\} $.
Choosing $\epsilon =  \|\rho - \sigma\|_1^{1/m}$, we conclude the proof. \qed

In the rest of this section we are concerned only with $n=2$. Given a mass density $\rho$ on $\R^2$
such that
\begin{equation}\label{defbnd}
\int_{\R^2} \log(e + |x|^2) \rho(x)\dd x  < \infty \,,
\end{equation}
the {\em Newtonian potential energy} of $\rho$ is given by
\begin{equation}\label{Newton}
{\mathcal U}[ \rho] := \int_{\R^2}\int_{\R^2} \rho(x) \log |x-y|\rho(y) \dd x \dd y\ .
\end{equation}

By the elementary inequality
$$\log_+ |x-y|_+ \leq   \log 2 + \log_+ |x| + \log_+|y|\ ,$$
(recall that $\log_+$ denotes the positive part of $\log$), the condition (\ref{defbnd}) ensures that the integral in (\ref{Newton}) is well-defined, though possibly with the value $-\infty$.

\begin{lm}    With $p$, $q$, $A$ and $B$ as in Definition~\ref{goset}, for all $0 < \epsilon < 1$,
there is an explicitly computable  constant $C$ depending only on  $\epsilon$, $p$, $q$, $A$ and $B$ such that  
and all $\rho,\sigma \in   {\mathcal M}_{2,p,q,A,B}$,
$$\left| {\mathcal U}[ \rho]  - {\mathcal U}[ \sigma] \right| \leq C \|\rho - \sigma\|_1^{1-\epsilon}\ .\ .$$
\end{lm}

\begin{proof} We define
$$ {\mathcal U}_+[ \rho]  :=  \int_{\R^2}\int_{\R^2} \rho(x) \log_+ |x-y|\rho(y) \dd x \dd y  \quad{\rm and}\quad 
 {\mathcal U}_-[ \rho]  :=  \int_{\R^2}\int_{\R^2} \rho(x) \log_- |x-y|\rho(y) \dd x \dd y\ .$$
Then using $\log_+|x|  \leq (1/r)|x|^{r}$ (see \eqref{eq:log r}) and likewise for $y$,
we obtain
\begin{eqnarray}
\left| {\mathcal U}_+[ \rho]  - {\mathcal U}_+[ \sigma] \right| &\leq& 
\int_{\R^2} |\rho(x) - \sigma(x)| \left(\int_{\R^2} (2 + \log_+|x| + \log_+|y|) [\rho(y) + \sigma(y)] \dd y \right)\dd x \nonumber\\
&\leq& C\int_{\R^2} |\rho(x) - \sigma(x)| \left( 1 + \frac1r |x|^r\right)\dd x\ .\nonumber
\end{eqnarray}
Hence, if $r \in (0,p)$, using H\"older's inequality we can estimate
$$\int_{\R^2} |\rho(x) - \sigma(x)|^{(p-r)/p}   |\rho(x) - \sigma(x)|^{r/p} |x|^r\dd x \leq \| \rho - \sigma\|_1^{(p-r)/p} 
\left(\int_{\R^2}[\rho(x)|x|^p +\sigma(x) |x|^p\dd x\right)^{r/p}\ .$$
Choosing $r = \epsilon p$, we get
\begin{equation}\label{new2}
\left| {\mathcal U}_+[ \rho]  - {\mathcal U}_+[ \sigma] \right|  \leq C \|\rho - \sigma\|_1^{1-\epsilon}\ ,
\end{equation}
for some constant $C$ depending only on $\epsilon$, $p$,  $A$ and $B$.

Next, for all $0< r <  2(q-1)/q$, by \eqref{eq:log r}
\begin{eqnarray}
\left| {\mathcal U}_-[ \rho]  - {\mathcal U}_-[ \sigma] \right| &\leq& 
\int_{\R^2} |\rho(x) - \sigma(x)| \left(\int_{\R^2}\log_-|x-y| [\rho(y) + \sigma(y)] \dd y \right)\dd x \nonumber\\
&\leq& \int_{\R^2} |\rho(x) - \sigma(x)| \left(\int_{\{|x-y| \leq 1\}}\frac{1}{r}|x-y|^{-r} [\rho(y) + \sigma(y)] \dd y \right)\dd x \nonumber\\
&\leq& \int_{\R^2} |\rho(x) - \sigma(x)| [\|\rho\|_q + \|\sigma\|_q]  \left(\int_{\{|y| \leq 1\}}\frac{1}{r}|y|^{-rq/(q-1)} \dd y  \right)^{(q-1)/q}\dd x \nonumber\\
&=& \|\rho - \sigma\|_1 [\|\rho\|_q + \|\sigma\|_q]  \left(\int_{\{|y| \leq 1\}}\frac{1}{r}|y|^{-rq/(q-1)} \dd y  \right)^{(q-1)/q}\ . \nonumber
\end{eqnarray}
The integral on the right is clearly finite for our choice of $r$, and we conclude that
\begin{equation}\label{new3}
\left| {\mathcal U}_-[ \rho]  - {\mathcal U}_-[ \sigma] \right|  \leq C \|\rho - \sigma\|_1 \ ,
\end{equation}
for some $C$ depending only on  $q$ and $B$.
Combining (\ref{new2}) and (\ref{new3}) we obtain the result. 
\end{proof}

\medskip
\noindent{\it Proof of Theorem~\ref{loghlsfunccont}:} The theorem follows directly from the  results proved  in this subsection. \qed

\end{document}